\documentclass[12pt,a4paper]{article}
\usepackage{eurosym}
\usepackage{amssymb}
\usepackage{amsmath,amsfonts,amsthm}
\usepackage{amstext,latexsym,bm}
\usepackage{graphicx, color}
\usepackage{amsmath}
\usepackage[english]{babel}
\usepackage{amsfonts}
\usepackage{amsthm,amscd}
\usepackage{amsmath}
\usepackage{amssymb}
\usepackage{amstext}

\newtheorem{theorem}{Theorem}
\newtheorem{corollary}[theorem]{Corollary}
\newtheorem{definition}[theorem]{Definition}
\newtheorem{example}[theorem]{Example}
\newtheorem{lemma}[theorem]{Lemma}

\newtheorem{proposition}[theorem]{Proposition}
\newtheorem{remark}[theorem]{Remark}

\begin{document}

\title{Grand-canonical Thermodynamic Formalism via IFS: volume, temperature,
gas pressure and grand-canonical topological pressure}
\author{ A. O. Lopes; E. R. Oliveira; W. de S. Pedra; V. Vargas}
\maketitle

\begin{abstract}
We consider here a dynamic model for a gas in which a variable number of
particles $N \in \mathbb{N}_0 := \mathbb{N} \cup \{0\}$ can be located at a
site. The dynamics are played by the shift acting on
$\Omega := \mathcal{A}^\mathbb{N}$, where 
$\mathcal{A} := \{1,2,...,r\}$. Introducing  the chemical potential $\mu$, we adapt the
concept of grand-canonical partition sum of thermodynamics of gases, considering a certain family of potentials $%
(A_N)_{N \in \mathbb{N}_0}$, $A_N:\Omega \to \mathbb{R}$. Extending classical thermodynamic formalism, we introduce the grand-canonical-Ruelle operator: $\mathcal{L}_{\beta, \mu}(f)=g$, when, $\beta>0,\mu<0,$ where

\medskip $\,\,\,\,\,\,\,\,\,\,\,\,\,\,g(x)= \mathcal{L}_{\beta, \mu}(f) (x)
=\sum_{N \in \mathbb{N}_0} e^{\beta \, \mu\, N }\, \sum_{j \in \mathcal{A}}
e^{- \,\beta\, A_N(jx)} f(jx). $ \medskip

We show the existence of the main eigenvalue, an associated eigenfunction,
and an eigenprobability for $\mathcal{L}_{\beta, \mu}^*$. We also consider the
concept of entropy for holonomic  probabilities on
$\Omega\times \mathcal{A}^{\mathbb{N}_0}$,  relating these items with the  problem of maximizing grand-canonical pressure. We briefly digress on a possible interpretation of the concept 
of topological pressure as related to the gas pressure of gas thermodynamics.	
\end{abstract}

\vspace{2mm}


\vspace{2mm}


\section{Introduction}

The study of the thermodynamics of gases with a non-specified number of
particles is a classical topic in Mathematical Physics (see Section 1.6 in
\cite{MR2777415} or Section 11.2.4 in \cite{MR3309581}). Here we will
investigate this type of problem from a dynamic perspective. That is the
search for statistical properties that can be obtained with the help of a
generalization of the Ruelle operator (which corresponds to the transfer operator of Statistical Mechanics)  of Thermodynamic Formalism (in the
sense of \cite{MR1085356} or \cite{Fan}). Concepts like volume, temperature, entropy, and
gas pressure arise naturally in thermodynamics when we introduce the number
of particles $N$ as a variable (see \cite{Cale} or Section 5.6 in \cite{Bena}).
The introduction of a negative constant $\mu $, called the chemical
potential, plays an important role in the convergence of the grand-canonical
partition sum (see (1.34) in \cite{MR2777415}). We analyze such kinds of
problems from a mathematical perspective and we leave the question of
physical relevance for a posterior investigation. The discussion in Section
3.2.4 in \cite{MR3309581} on the topic of probabilities for particle
distributions is quite enlightening. We would like to emphasize that the
postulates of equilibrium thermodynamics of gases are an issue subject to
controversy (see Section \ref{Tgas}).

In classical Thermodynamic Formalism, in general, results avoid taking into
account  this variable number of particles.

We are interested in the mathematical formulation of physical problems in
equilibrium from a dynamic perspective. Time does not occur as a variable in
thermodynamic equations. When we allude \textit{a mathematical formulation
in a dynamical setting}, by this, we mean problems related to the action of
the shift $\sigma$ on the symbolic space $\{1,2,...,r\}^\mathbb{N}$; this is associated with translation on the
one-dimensional lattice and is not related to time.

We are interested in the statistics of the number of particles: any number $%
N \in \mathbb{N}_0:= \mathbb{N} \cup \{0\}$ particles can be in one site. Therefore, in principle,
it is natural to consider an IFS with a countable number of functions (and
with weights), but it is possible (for part of the results we consider) to translate it, after some work,  to the case of a finite
one (and then results from \cite{LO} can be used). 

We will consider a family of potentials $A_N: \{1,2,...,r\}^\mathbb{N} \to \mathbb{R}$, $N \in \mathbb{N}_0$, and we are interested in
equilibrium states. In the IFS setting instead of shift-invariant
probabilities it is natural to consider holonomic probabilities, as
described in \cite{LO} (see Definition \ref{lolo}). All this will be
carefully described in Section \ref{Boss}.

We introduce what we call the grand-canonical-Ruelle operator  (see \eqref{polu}) and we show a version of the Ruelle Theorem (about
eigenfunctions and eigenvalues), which is presented as our main result in
Theorem \ref{thm: equili_IFS_fin}.   A value $\beta>0$ plays an importan role in the theory, and $\beta=\frac{1}{T}$, where $T>0$ is temperature.
The main eigenvalue will be called the
grand-canonical eigenvalue.
We will assume some mild conditions  for potentials $A_N$, $N \in \mathbb{N}_0$, in order to control the behavior of the grand-canonical operator.  The family of potentials $A_N$ can growth, for instance, like $N$ (but not necessarily like that).

We present at the end of the paper an example which is simple but illustrates the connection with a classical model in Statistical Mechanics, where the number of particles  $N$  ranges in the set $\mathbb{N}_0$  (see details in  Example \ref{Ham-cte}). Consider the value $\mu < 0$, which is called the chemical potential. This plays an essential  role in questions involving the 
grand-canonical formalism. Assume that $E > 0$ represents the energy of a particle, and then $N\, E$ is  the  total  energy of  $N$ particles. Consistent with the development of the theory to be developed here, given $\beta>0$, we should take  the Hamiltonian $A_N = N\,E$, $N\in \mathbb{N}_0$, and consider the expression

$$\sum_{N \in \mathbb{N}_0}e^{\beta\, N\,\mu} e^{- \beta N\,
	E}= \frac{1}{ 1-
	e^{\beta (\mu- E)}} > 0 ,$$
which is called the grand-canonical partition sum.
The value $\mu<0$ helps in the convergence of the series. The above example does not contemplate dynamics and is quite particular because each $A_N$ is a constant function.
The results we get here describe an interplay between ergodic theory and the thermodynamic of gases.

\medskip
The thermodynamic formalism, while related to the principle of maximizing topological pressure, shares a basic principle with gas dynamics: the second law of thermodynamics; nature maximizes entropy at equilibrium. In this sense, the respective concepts of entropy must correspond. From the perspective of analyzing grand canonical questions from a dynamical viewpoint (which involves dealing with an indefinite number of particles), there is a need to introduce a natural concept of entropy; this is one of our objectives. The grand-canonical-Ruelle operator plays an important role here.
\medskip

We recall that for a (uniformly) continuous function $f: \Omega \rightarrow \mathbb{R}$ the modulus of continuity of $f$ is
\begin{equation} \label{Didu0}
\omega_f(t)=\sup _{d(x, y) \leq t} f(x)- f(y).
\end{equation}
The function $f$ is called Dini-continuous if
$$
\int_0^1 \frac{\omega_f(t)}{t} dt< \infty.
$$
An equivalent condition (see \cite{Ste01}) is the following: for some $c \in(0,1)$,
\begin{equation} \label{Didu}
\sum_{i=1}^{\infty} \omega_f\left(c^i \right)<\infty
\end{equation}
(since $\Omega$ is compact and has diameter equal to $1$). It is easy to see that $f$ is Lipschitz (resp. $\alpha$-H\"{o}lder) if, and only if $\omega_f\left(t\right) \leq \operatorname{Lip}(f)\, t$ (resp. $\omega_f\left(t\right) \leq \operatorname{Hol}(f)\, t^{\alpha}$). Thus, both classes are contained in the class of Dini continuity.

The grand-canonical potential
$$\psi(y):= \ln\left(\sum_{N \in \mathbb{N}_0}  e^{-\beta [ A_N(y) - \mu N]}\right),$$
will play an important role in our reasoning (in Section \ref{ense} we present  appropriate conditions on the $A_N$). It will be required that $\psi$ is at least Dini continuous (see \eqref{Didu0} and \eqref{Didu} for the definition).

In Corollary \ref{ana} we will show an analytical dependence of the grand-canonical eigenvalue in $\psi$ assuming more regularity on $\psi$.

In Section \ref{ense} we recall some classical results in thermodynamic
formalism and we introduce the dynamical canonical ensemble in this context. Later, we will analyze the main properties of the grand-canonical-Ruelle operator, the concepts of entropy for holonomic probabilities (see
Definitions \ref{lolo} and \ref{entropy}), and also the grand-canonical topological pressure
(see \eqref{fip} and item a) in Theorem \ref{thm: equili_IFS_fin} and also
\eqref{ainf}). 

Given a family of potentials  $A_N:\Omega= \{1,...,r\}^{\mathbb{N}} \to \mathbb{R}$, $N \in \mathbb{N}_0 $, $\beta>0$, and $%
\mu<0,$  satisfying a Dini condition, the grand-canonical-Ruelle operator $f \to\mathcal{L}_{\beta, \mu}(f)=g$,  is given by
\begin{equation}  \label{polu}
g(x)= \mathcal{L}_{\beta, \mu}(f) (x) =\sum_{N \in \mathbb{N}_0} e^{\beta \,
\mu\, N }\, \sum_{j \in \mathcal{A}} e^{- \,\beta\, A_N(jx)} f(jx).
\end{equation}

We denote by $\Phi=(A_N)_{N \in \mathbb{N}_{0}}$ the family of potentials.


Note that the points of the form $j\,x$, $j \in
\mathcal{A}$, describe the set of solutions $y$ of $\sigma(y)=x$. Then, the
operator $\mathcal{L}_{\beta, \mu}$ is {\bf dynamically} defined; it corresponds
to the classical transfer operator of Statistical Mechanics but for a dynamical setting.


\smallskip

For the benefit of the mathematical reader, we will briefly describe some
basic properties of the thermodynamics of \textit{ideal} gases in Section
\ref{Tgas}. Reading this section is not necessary for understanding the
mathematical reasoning followed in the previous sections. The objective is
only to show the motivation that led us to analyze the problems that were
proposed.

In Remark~\ref{rel} we will investigate a possible interpretation of the
terminology topological pressure in a comparison with the concept of gas
pressure, which originated from the postulates of the theory that analyzes
gases confined under certain variable walls and at a certain temperature.

In a related work, the authors consider in \cite{LR} non-equilibrium and the
second law of thermodynamics in Thermodynamic Formalism. In \cite{LW} it is
presented a brief account of Thermodynamics, Statistical Physics, and their
relation to the Thermodynamic Formalism of Dynamical Systems.

The study of Thermodynamic Formalism for symbolic spaces with an infinite
countable alphabet (the set $\mathcal{A}=\mathbb{N}$) is the topic of \cite%
{Sarig}, \cite{BBE} and \cite{FV}; but a different class of problems is
considered there.

Results for IFS using conformal branches appear in \cite{Mih} but it is also a different setting compared to ours.
\medskip

 Conclusion: The classical study of the grand-canonical
partition sum in the thermodynamics of gases considers an indefinite number
of particles $N$, a Hamiltonian $A_{N}$, and the chemical potential $\mu $,
which is negative in order to ensure convergence of the associated
sum. 

We consider the corresponding problems on the symbolic dynamical
setting considering a Dini family of potentials $(A_N)_{N \in \mathbb{N}
_0}$, $A_N:\Omega \to \mathbb{R}$. We introduce the grand-canonical-Ruelle
operator: $\mathcal{L}_{\beta, \mu}$, $\beta>0,\mu<0$ (as defined
in \eqref{polu}), and we can get concepts like discrete-time entropy
(and also the pressure problem associated with such entropy). Our main results
will be obtained from adapting well-known
properties of the Thermodynamic Formalism for IFS with weights to our
dynamical setting. The naturally associated Gibbs probability is not 
shift invariant (it is holonomic).  One of our main results
is Theorem \ref{thm: Fan_Boson}, which shows the existence of eigenfunctions
and eigenprobabilities (a key step for analyzing questions related to
maximizing pressure). In the  variational problem of  grand-canonical topological pressure (see Theorem
\ref{thm: equili_IFS_fin})  the holonomic probabilities play an important (and natural) role due to the structure of the  IFS setting; indeed,  it is required a special (and natural) concept of entropy which is described by Definition \ref{entropy}.

We follow two different
lines of reasoning, the first one is modeling the problem via a finite IFS
with weights 
see Section \ref{secfi}). Alternatively, we consider an
infinitely countable IFS setting (see Section \ref{outy}).
Note that in Section 3.1 we get probabilities on  $\Omega \times \mathcal{A}^{\mathbb{N}_0}$ and in Section 3.2 we get probabilities on $\Omega$.

\medskip

\section{A brief review of Classical Thermodynamic Formalism}

\label{ense}

Consider a finite alphabet $\mathcal{A} := \{1, ... , r\}$ and the \textit{%
shift map} $\sigma((x_n)_{n \in \mathbb{N}}) := (x_{n+1})_{n \in \mathbb{N}}$
acting on the \textit{symbolic space} $\Omega := \mathcal{A}^\mathbb{N}$
which is equipped with the metric (which makes ${\rm diam}(\Omega)<1$)
\begin{equation*}
d(x, y) := \left\{
\begin{array}{ll}
2^{-\min\{n \in \mathbb{N} : x_n \neq y_n \}}, & x \neq y; \\
0, & x=y.%
\end{array}
\right.
\end{equation*}
The dynamical system $(\Omega, \sigma)$ is widely known in the mathematical
literature as the \textit{full-shift} on the alphabet $\mathcal{A}$. We
denote the set of \textit{continuous functions} from $\Omega$ into $\mathbb{R%
}$ by $\mathrm{C}(\Omega)$ and we use the notation $\mathrm{C}^+(\Omega)$
for the corresponding cone of \textit{positive continuous functions}. We
also denote the set of \textit{Lipschitz continuous functions} from $\Omega$
into $\mathbb{R}$ by $\mathrm{Lip}(\Omega)$ and we use the notation $\mathrm{%
Lip}(f)$ for the \textit{Lipschitz constant} of $f \in \mathrm{Lip}(\Omega)$%
. Besides that, we denote the set of \textit{Borel probability measures} on $%
\Omega$ by $\mathcal{M}_1(\Omega)$ and we use the notation $\mathcal{M}%
_\sigma(\Omega)$ for the set of Borel \textit{$\sigma$-invariant probability
measures} on $\Omega$.

Here, we consider a system describing the dynamical behavior of a classical
gas one-dimensional lattice composed of $N$ particles at temperature $T$,
which are contained in a region with volume $V$. It is natural to introduce a parameter $\beta$ in such a way that  satisfies
relation
\begin{equation}  \label{Boltzmann-constant}
\beta := \frac{1}{k_B \, T} \;,
\end{equation}
where $k_B \sim 1.38066 \times 10^{-23} J/K$ is the so-called \textit{%
Boltzmann's constant} (see (1.2) in \cite{MR2777415} for details). For practical purposes, people{\tiny } sometimes incorporate  the constant $k_B $ in the temperature $T$, when dealing with the expression of $\beta$.

We assume that the number of particles ranges on the set $\mathbb{N}_{0} :=
\mathbb{N} \cup \{0\}$ and we consider a potential $A : \Omega \times
\mathbb{N}_0 \to \mathbb{R}$ which is Lipschitz continuous w.r.t. the first
variable. It is not difficult to check that the potential $A$ induces a
family of potentials $\Phi = (A_N)_{N \in \mathbb{N}_0}$, where $A_N :=
A(\cdot, N)$ for each $N \in \mathbb{N}_0$. In fact, the last assumption
guarantees that $A_N \in \mathrm{Lip}(\Omega)$.


We consider first the case where $N$, the number of particles, is a fixed
natural number. This corresponds to just considering a classical Ruelle operator (as in \cite{MR1085356}). Given
$N \in \mathbb{N}_0$ and some $\beta > 0$ satisfying the expression in %
\eqref{Boltzmann-constant}, we consider the \textit{Ruelle operator} $%
\mathcal{L}_{N, \beta}$ associated  to a   Dini potential $A_N$, as the one given by the equation
\begin{equation}  \label{kry1}
\mathcal{L}_{N, \beta}(f)(x) := \sum_{\sigma(y)=x} e^{-\,\beta\, A_N(y)}
f(y) = \sum_{j \in \mathcal{A}} e^{-\,\beta\, A_N(jx)} f(jx), \forall x \in
\Omega.
\end{equation}

It is well known (when  $A_N$  satisfy a Dini condtion)  that for each pair $N, \beta$, there are a main \textit{%
eigenvalue} $\lambda_{N, \beta} > 0$ and an \textit{eigenfunction} $f_{N,
\beta} \in \mathrm{Lip}(\Omega)$ for the operator $\mathcal{L}_{N, \beta}$   (see for instance  \cite{Fan} or \cite{Fan0}).

Given a continuous potential $A_N:\Omega \to \mathbb{R}$ and $\beta>0$, we  can define the dual operator $\mathcal{L}_{N, \beta}^*$ acting on the space of the Borel finite measures (see \cite{MR1085356}).

We denote by $\nu_{N, \beta} \in \mathcal{M}_1(\Omega)$ the \textit{%
eigenprobability} of the operator $\mathcal{L}_{N, \beta}^*$ associated to $%
\lambda_{N, \beta}$ and by $\rho_{N, \beta} \in \mathcal{M}_\sigma(\Omega)$
the \textit{equilibrium state} for the potential $- \beta A_N$ which, up to
a normalization, is of the form $\rho_{N, \beta} = f_{N, \beta}\;
\nu_{N, \beta}$ (see \cite{MR1085356} or \cite{Lop1} for details).

Given $x \in \Omega$ and $N \in \mathbb{N}_0$, we define the \textit{$N$%
-canonical partition} for the iterate $n \in \mathbb{N}$ calculated at the
point $x\in \Omega$ by
\begin{equation}  \label{kry3}
Z_N^n(\beta)(x) := \mathcal{L}_{N, \beta}^n (\mathit{1})(x) \;,
\end{equation}
where $\mathcal{L}_{N, \beta}^{n+1}(f) = \mathcal{L}_{N, \beta}(\mathcal{L}%
_{N, \beta}^n(f))$ for each $n \in \mathbb{N}$.

The pointwise limit $\lim_{n \to \infty}\frac{1}{n}%
\log(Z_N^n(\beta)(x))$, which is independent of $x$ (see next Lemma), plays  here the role of the so-called configurational
partition sum appearing at (1.17) on page 7 of \cite{MR2777415}.

The study of the properties of an individual Transfer operator $\mathcal{L}%
_{N, \beta}$,  for $N$ and $\beta$ fixed, it is not suitable for the case where the
number of particles $N$ ranges in the set of natural numbers (which is the
goal of the next section).

The next lemma is well-known in Thermodynamical Formalism and we will not
present a proof (see \cite{MR1085356}).

\begin{lemma}
\label{lemma-pressure} For fixed $N$ and $\beta>0$, the pointwise limit $\lim_{n \to \infty}\frac{1}{n}%
\log(Z_N^n(\beta)(x))$ exists and it is equal to $\log(\lambda_{N, \beta})$.
In particular, it is independent of the choice of $x \in \Omega$.
\end{lemma}

We call \textit{$N$-Topological Pressure} for $\beta$ to the observable
satisfying
\begin{equation*}
P_N(\beta) = P(-\beta A_N) := \lim_{n \to \infty} \frac{1}{n} \log
(Z_N^n(\beta)) = \log(\lambda_{N, \beta}) \;.
\end{equation*}

In this way, by Lemma \ref{lemma-pressure} we obtain that $Z_N^n(\beta) \sim
\lambda_{N, \beta}^n$. Moreover, one can show that the following expression
holds true
\begin{equation}  \label{kry0}
P_N(\beta) = P(-\beta A_N) = \sup_{\rho \in \mathcal{M}_\sigma(\Omega)} %
\Bigl\{ h(\rho) - \beta\, \int A_N \,d \rho \Bigr\},
\end{equation}
where $h(\rho)$ is the \textit{Kolmogorov-Sinai entropy} of $\rho$ and $%
\mathcal{M}_\sigma(\Omega)$ denotes the set of $\sigma$-invariant
probabilities (for details see \cite{MR1085356}).

For a grand-canonical version of topological pressure, we will need a different version of \eqref{kry0} due to the fact that we have to consider probabilities such that the concept of 
Kolmogorov-Sinai entropy does not apply (see expression \eqref{fie}).


The above computation implies that $\displaystyle\lim_{N \to \infty}
P_N(\beta) = - \infty$ and $\displaystyle\lim_{N \to \infty} \lambda_{N,
\beta} = 0$. Furthermore, in \cite{MR1085356} (see also \cite{MR3852182}),
the authors prove that
\begin{align*}
\frac{\partial }{\partial \beta} \log \lambda_{N, \beta}|_{\beta=\beta_0} &=
\frac{\partial }{\partial \beta} P_N(\beta)|_{\beta=\beta_0} \\
&= \frac{\partial }{\partial \beta} P(-\beta A_N)|_{\beta=\beta_0} = -\int
A_N d \rho_{N, \beta_0} \;.
\end{align*}

So, by the above formula, we get
\begin{equation}  \label{derivative-Z_N}
\frac{1}{ \lambda_{N, \beta_0}}\frac{\partial \lambda_{N, \beta}}{\partial
\beta}|_{\beta=\beta_0} = -\int A_N d\rho_{N, \beta_0} \;.
\end{equation}

Above we described the classical dynamical properties of the individual
transfer operator $\mathcal{L}_{N, \beta}$. In the next section, we will use
properties of IFS Thermodynamical Formalism to address the analogous issue
for the case of a variable number of particles, where it is necessary to
consider a countable number of classical Ruelle operators (each one
indexed by the number $N$ of particles). We believe the material presented
on the present section will help the reader to understand the reasoning of
the next one.

\section{A grand-canonical Thermodynamic Formalism}

\label{Boss}

Here we consider a variable number of particles. In order to do that, we
consider  linear operators involving the new variable $N$ (describing the
number of particles), which is defined in the following way: given a family
of  potentials $\Phi = (A_N)_{N \in \mathbb{N}_0}$ (which play the role of Hamiltonians),
a chemical potential $\mu < 0$ and a value $\beta > 0$ satisfying the
expression in \eqref{Boltzmann-constant}, the \textit{grand-canonical-Ruelle
operator} $\mathcal{L}_{\beta, \mu}$ is defined as the operator assigning to
each $f \in \mathrm{C}(\Omega)$ the function
\begin{equation}  \label{BRO}
\mathcal{L}_{\beta, \mu}(f) (x) :=\sum_{N \in \mathbb{N}_0} e^{\beta \,
\mu\, N }\, \sum_{a \in \mathcal{A}} e^{- \,\beta\, A_N(ax)} f(ax) = \sum_{N
\in \mathbb{N}_0}\,e^{\beta \, \mu\, N }\,\mathcal{L}_{N,\beta}(f)(x), \;
\end{equation}
for any $x \in \Omega$.

In order to get convergence in the above sum we need some hypotheses: for fixed $\mu$
we will assume that the family of potentials $A_N$, $N \in \mathbb{N}$, is admissible, that is, we assume that for any $x\in \Omega$ the sums
$$\sum_{N \in \mathbb{N}_0} e^{ \beta\, n\, \mu} \sum_{a \in \mathcal{A}} e^{-\beta\,  A_N (  a\, x)}<\infty.$$
Then, it follows for any $f$ that
$$ ||\mathcal{L}_{\beta, \mu}(f)||_\infty \leq ||f||_\infty\,  \sum_{N \in \mathbb{N}_0} e^{ \beta\, n\, \mu} \sum_{a \in \mathcal{A}} e^{-\beta\,  A_N (  a\, x)}<\infty.$$

Once \eqref{BRO} is well defined, we can ask about the existence of eigenvalues,
eigenfunctions for $ \mathcal{L}_{\beta, \mu}$, and also holonomic probabilities for the IFS pressure. Our
main goal is to represent (after some work) the operator $\mathcal{L}_{\beta, \mu}$ as the
transfer operator of a standard IFS with weights for which the thermodynamic
formalism is already known from the literature. One can follow two different
lines of reasoning, the first one is modeling the problem via a finite IFS
with weights (see Section \ref{secfi}). We have to show that our model fits
the hypothesis of \cite{LO} and \cite{Fan}. Alternatively, one could use an
infinitely countable IFS (see Section \ref{outy}), this is fine as we will see,  but it brings technical
difficulties and some limitations as will be explained. \medskip

\subsection{Transferring the problem to the case of a finite IFS with weights.}

\label{secfi}

In this section, we introduce an IFS with weights in such a way that its
associated transfer operator coincides with the grand-canonical-Ruelle
operator $\mathcal{L}_{\beta, \mu}$. By showing that the weights satisfy the
necessary regularity conditions we will use Fan's Theorem (see \cite{Fan},
Theorem 1.1) to obtain a positive eigenfunction for $\mathcal{L}_{\beta,
\mu} $. Note that the weights are not periodic. Once we have this positive eigenfunction associated with the
spectral radius of $\mathcal{L}_{\beta, \mu}$ we can introduce the
thermodynamical formalism based on holonomic measures according to \cite{LO}
or \cite{CO17}. We recall the ideas of variational entropy and topological
pressure based on holonomic probabilities. Finally, we will show that is
possible to build a variational principle and show the existence of
holonomic equilibrium states.

\medskip

Given $r \geq 2$, consider the \textit{IFS} $\mathcal{R}:=(\Omega,
\phi_{j})_{j \in \mathcal{A}}$ where $\mathcal{A}:=\{1,...,r\}$, $%
\phi_{j}(x)=jx$ is the \textit{mnemonic representation} for the sequence $%
(j, x_1, x_2, ...)$,  where  $x=(x_1, x_2, ...) \in \Omega=\mathcal{A}^{\mathbb{N}%
}$. Of course, this IFS is contractive w.r.t. the distance introduced in $%
\Omega$. Moreover, ${\rm Lip} (\phi_{j}) = \frac{1}{2}$ for all $j\in \mathcal{A}$.

Given a family of continuous functions $q_{j}: \Omega \to \mathbb{R}, \; j
\in \mathcal{A}$ we say that $\mathcal{R}:=(\Omega, \phi_{j}, q_{j})_{j \in
\mathcal{A}}$ is an \textit{IFS with weights}. In the particular case where $%
q_{j}(x) \geq 0$ and $\sum_{j \in \mathcal{A}}q_{j}(x)=1$, for all $x
\in\Omega$, it is called an \textit{IFS with probabilities}. According to
\cite{Fan}, an IFS with weights where the maps are contractions and the
weights are non-negative is called a \textit{contractive system}.

In this setting the \textit{transfer operator} associated to $\mathcal{R} =
(\Omega, \phi_{j}, q_{j})_{j \in \mathcal{A}}$ is a map $B_{q}: \mathrm{C}%
(\Omega) \to \mathrm{C}(\Omega)$ given by:
\begin{equation}  \label{eq:op_transf_IFS_finit}
B_{q}(g)(x) := \sum_{j \in \mathcal{A}} q_{j}(x) g(\phi_{j}(x)), \forall x
\in \Omega,
\end{equation}
for any $g \in \mathrm{C}(\Omega)$.

The next lemma shows how to pick the right weights $q_{j}$ in order to
obtain the equality $B_{q} = \mathcal{L}_{\beta, \mu}$.

Let $\psi: \Omega \rightarrow \mathbb{R}$ be the grand-canonical potential
\begin{equation} \label{hert} \psi(y):= \ln\left(\sum_{N \in \mathbb{N}_0}  e^{-\beta [ A_N(y) - \mu N]}\right),
\end{equation}
where the family of potentials $A_N$ satisfies some prescribed conditions so that the formal series converges.

We recall that for a function$f: \Omega \rightarrow \mathbb{R}$ the modulus of continuity of $f$ is
\begin{equation} \label{Didu0}
\omega_f(t)=\sup _{d(x, y) \leq t} f(x)- f(y).
\end{equation}
The function $f$ is called Dini-continuous if
$$
\int_0^1 \frac{\omega_f(t)}{t} dt< \infty.
$$
An equivalent condition (see \cite{Ste01}) is the following: for some $c \in(0,1)$,
\begin{equation} \label{Didu}
\sum_{i=1}^{\infty} \omega_f\left(c^i \right)<\infty
\end{equation}
(since $\Omega$ is compact and has diameter equal to $1$). It is easy to see that $f$ is Lipschitz (resp. $\alpha$-H\"{o}lder) if, and only if $\omega_f\left(t\right) \leq \operatorname{Lip}(f)\, t$ (resp. $\omega_f\left(t\right) \leq \operatorname{Hol}(f)\, t^{\alpha}$). Thus, both classes are contained in the class of Dini continuity.

In our reasoning, it will be required to assume the  hypothesis  assuring that  \eqref{hert} is at least Dini continuous (see condition b) in the next Lemma).

\begin{lemma}
\label{lem:equal_operat_finite} Consider  the family of  potentials $\Phi=(A_N)_{N \in \mathbb{N}_{0}}$ and the weights
\begin{equation*}
q_{j}(x):= e^{\psi(\phi_{j}(x))} > 0, \; j \in \mathcal{A}.
\end{equation*}
\begin{itemize}
  \item[a)] If $\Phi$  satisfy
  $$ \liminf_{N \to \infty} \frac{A_N(x)}{N} > \mu, \; \forall x \in \Omega,$$
  then the contractive system $\mathcal{R} = (\Omega, \phi_{j},
q_{j})_{j \in \mathcal{A}}$ is well defined and $B_{q}(g)= \mathcal{L}_{\beta, \mu}(g),$ for any $g \in \mathrm{C}(\Omega)$;
  \item[b)]  Suppose that
  \begin{equation}  \label{strong super}
\exists \,\varepsilon>0, \delta \geq 0 \text{ s.t. } \; A_N(x) > (\mu + \varepsilon) N + \delta, \; \forall x \in \Omega, \forall N \in \mathbb{N}_0.
\end{equation}
  If each $A_N$ is Dini continuous and $$\limsup_{i \to \infty} \left(\sum_{N \in \mathbb{N}_0}  \omega_{A_N}(c^{i}) \; (e^{-\beta\, \varepsilon})^{N}\right)^{1/i} <1,$$
  then $\psi$ is Dini continuous. In particular, if the family of potentials is uniformly Lipschitz, that is,  $\operatorname{Lip}(A_N) \leq M$, then $\operatorname{Lip}(\psi) \leq \beta\, M$ (and so $\operatorname{Lip}(\ln(q_{j})) \leq \frac{\beta\, M}{2}$).
\end{itemize}
\end{lemma}

\begin{proof}
   (a) The proof follows easily from the commutativity of the summation in the formula for  $\mathcal{L}_{\beta, \mu}$ if we prove that for each $j \in \mathcal{A}$ the positive series $\sum_{N \in \mathbb{N}_0}  e^{-\beta [ A_N(jx) - \mu N]}$ is convergent. Consider applying the root test:
   $$\limsup_{N \to \infty} \sqrt[N]{e^{-\beta \; [ A_{N}(jx) - \mu (N)]}}= e^{-\beta\; \liminf_{N \to \infty} [ \frac{1}{N}A_{N}(jx) - \mu ]} <1 $$
   if and only if $\liminf_{N \to \infty} \frac{1}{N}A_{N}(x) >  \mu$, which is our hypothesis. 

   \smallskip
   (b) For the second part, we  notice that our assumption $A_N(x) > (\mu + \varepsilon) N + \delta$ ensures that $\liminf_{N \to \infty} \frac{A_N(x)}{N} > \mu, \; \forall x \in \Omega$. Thus, from (a), $\psi$ and $q_{j}(x)$ are well defined.\\
   For each $x \in \Omega$ we define a probability $\nu_x$ over $\mathbb{N}_{0}$ by the formula
$$\int_{\mathbb{N}_{0}} g(N) d\nu_x(N):= \frac{1}{e^{\psi(x)}} \sum_{N \in \mathbb{N}_0}  g(N) \; e^{-\beta [ A_N(x) - \mu N]},$$
for any continuous function $g: \mathbb{N}_{0} \to \mathbb{R}$.\\
A consequence from our assumption \eqref{strong super}, is that for any $x \in \Omega$ we have
$$ A_N(x) > (\mu+ \varepsilon) N + \delta \Leftrightarrow e^{-\beta [ A_N(x) - \mu N]} < e^{-\beta [ \varepsilon N + \delta]}.$$
Moreover,
$$e^{\psi(x)}= \sum_{N \in \mathbb{N}_0}  e^{-\beta [ A_N(x) - \mu N]}> e^{-\beta [ A_0(x) - \mu 0]}= e^{-\beta  A_0(x)}> e^{-\,\beta  \|A_0\|_{0}}=\gamma >0,$$
thus,
$$\int_{\mathbb{N}_{0}} g(N) d\nu_x(N)< \frac{1}{\gamma} \sum_{N \in \mathbb{N}_0}  g(N) \; e^{-\beta [ \varepsilon N + \delta]}$$
or
\begin{equation}\label{eq: estima integral}
\sup_{x \in \Omega} \int_{\mathbb{N}_{0}} g(N) d\nu_x(N) < \frac{1}{\gamma e^{\beta  \delta}} \sum_{N \in \mathbb{N}_0}  g(N) \; (e^{-\beta\, \varepsilon})^{N}
\end{equation}
for any continuous function $g: \mathbb{N}_{0} \to \mathbb{R}$.\\
We notice that,  for any $c \in(0,1)$ and $x,y \in \Omega$ with $d(x,y) < c$ we have
   $$\psi(x)- \psi(y) = - \ln\left(\frac{\sum_{N \in \mathbb{N}_0}  e^{-\beta [ A_N(y) - \mu N]}}{\sum_{N \in \mathbb{N}_0}  e^{-\beta [ A_N(x) - \mu N]}} \right) = $$
   $$= - \ln\left(\frac{\sum_{N \in \mathbb{N}_0}
   e^{-\beta [ A_N(y) - A_N(x)]} \; e^{-\beta [ A_N(x) - \mu N]}}{e^{\psi(x)}} \right) = $$
   $$= - \ln\left(\int_{\mathbb{N}_{0}} e^{-\beta [ A_N(y) - A_N(x)]} d\nu_x(N)\right) \leq \int_{\mathbb{N}_{0}}- \ln\left( e^{-\beta [ A_N(y) - A_N(x)]}\right) d\nu_x(N)= $$
   $$=\beta \int_{\mathbb{N}_{0}} [A_N(y) - A_N(x)]  d\nu_x(N)  \leq  \beta \int_{\mathbb{N}_{0}} \omega_{A_N}(c)  d\nu_x(N).$$
   In the above inequality, we used the fact that $-\ln(\cdot)$ is a convex function and $\nu_x$ is a probability, so  Jensen's inequality holds.\\
   Thus
   $$ \sum_{i=1}^{\infty} \omega_{\psi} \left(c^i\right) \leq \beta \sum_{i=1}^{\infty} \sup_{x \in \Omega} \int_{\mathbb{N}_{0}} \omega_{A_N}(c^{i})  d\nu_x(N).$$
   From \eqref{eq: estima integral} we get
   $$ \sup_{x \in \Omega} \int_{\mathbb{N}_{0}} \omega_{A_N}(c^{i}) d\nu_x(N) < \frac{1}{\gamma e^{\beta  \delta}} \sum_{N \in \mathbb{N}_0}  \omega_{A_N}(c^{i}) \; (e^{-\beta\, \varepsilon})^{N},$$
   thus
   $$ \sum_{i=1}^{\infty} \omega_{\psi} \left(c^i\right) \leq \frac{\beta}{\gamma e^{\beta  \delta}}    \sum_{i=1}^{\infty} \left(\sum_{N \in \mathbb{N}_0}  \omega_{A_N}(c^{i}) \; (e^{-\beta\, \varepsilon})^{N}\right).$$
   By the root test, this non negative series is convergent if
   $$\limsup_{i \to \infty} \left(\sum_{N \in \mathbb{N}_0}  \omega_{A_N}(c^{i}) \; (e^{-\beta\, \varepsilon})^{N}\right)^{1/i} <1$$
   which is a pretty easy condition to fulfill, since for each $N$ we have $\omega_{A_N}(c^{i}) \to 0$ when $i \to \infty$, because $A_N$ is Dini continuous, we just need to have some controlled growing with respect to $N$ (see remarks after).\\
   In order to conclude our proof we notice that, if $\operatorname{Lip}(A_N) \leq M$ then we get $\omega_{A_N}(c^{i}) \leq  M c^{i}, \forall N \in \mathbb{N}_{0}$, then from the previous computations
   $$ \omega_{\psi} \left(c^i\right) \leq \beta  \sup_{x \in \Omega} \int_{\mathbb{N}_{0}} \omega_{A_N}(c^{i})  d\nu_x(N) \leq \beta  \sup_{x \in \Omega} \int_{\mathbb{N}_{0}} M c^{i}  d\nu_x(N)= \beta\, M\; c^{i},  $$
   that is $\operatorname{Lip}(\psi) \leq \beta\, M$.
\end{proof}

\begin{remark}
We notice that in Lemma~\ref{lem:equal_operat_finite},  item (a),  the condition
\begin{equation*}
\liminf_{N \to \infty} \frac{1}{N}A_{N}(x) > \mu
\end{equation*}
is only sufficient. As a matter of fact, if we assume that $(A_N)_{N \in
\mathbb{N}_{0}}$, is an increasing sequence of functions then we can use
Dalembert's convergence test:
\begin{equation*}
\limsup_{N \to \infty}\frac{e^{-\beta [ A_{N+1}(jx) - \mu (N+1)]}}{e^{-\beta
[ A_N(jx) - \mu N]}}=\limsup_{N \to \infty} e^{-\beta [ A_{N+1}(jx) -
A_N(jx) - \mu]}<1,
\end{equation*}
if and only if $\limsup_{N \to \infty} A_{N+1}(jx) - A_N(jx) - \mu >0$,
which is the case because $-\mu >0$ and $A_{N+1}(jx) > A_N(jx)$ by
hypothesis.
\end{remark}

\begin{example} \label{ex:energy}
   We  choose $A_{N}(x):=N\, E(x)\geq 0$ for any $x \in \Omega$, where the energy $E:\Omega \to \mathbb{R}$ is Lipschitz (or even Dini) continuous. Taking $\delta=0$, $0<\varepsilon < \min E(x) -\mu$ we get
   $$  A_N(x) > (\mu + \varepsilon) N + \delta, \; \forall x \in \Omega, \forall N \in \mathbb{N}_0.$$
   Moreover, each $A_N$ is Dini continuous and $\omega_{A_N}(c^{i}) \leq N \omega_{E}(c^{i}) \leq  N\, \operatorname{Lip}(E) \, c^{i}$ thus
   $$\limsup_{i \to \infty} \left(\sum_{N \in \mathbb{N}_0}  \omega_{A_N}(c^{i}) \; (e^{-\beta\, \varepsilon})^{N}\right)^{1/i} \leq $$ $$\leq c\; \limsup_{i \to \infty} \left(\operatorname{Lip}(E) \;\sum_{N \in \mathbb{N}_0} N\, (e^{-\beta\, \varepsilon})^{N}\right)^{1/i} <1.$$
   Thus, from Lemma~\ref{lem:equal_operat_finite},  item (b), $\psi$ is Dini continuous.
\end{example}

{\begin{example} \label{ex:superlinear non Lipsc}
   In this example, we provide a different construction of a Dini continuous grand-canonical potential $\psi$, where each potential  $A_N:\{0,1\}^\mathbb{N} \to \mathbb{R}$ is Lipschitz continuous (but with unbounded Lipschitz constant).

   We  choose $A_{N}(x):=\theta(N)\sum_{j=1}^{+\infty} \frac{x_{j}}{2^j}$ for any $x \in \Omega$, where $\theta(N)$ is a given sequence of nonnegative real numbers.  Notice that $A_{N}(x)=B_{N}(t)=(B_{N}\circ\pi)(x)$, where $t=\pi(x):=\sum_{j=1}^{+\infty} \frac{x_{j}}{2^j}$ and $B_{N}(t):=\theta(N) \, t,\; t \in [0,1]$. As $\pi$ preserves distance we just need to check if
   $$\psi(t):=\ln\left(\sum_{N=0}^{\infty}  e^{-\beta [ B_N(t) - \mu N]}\right)=\ln\left(\sum_{N=0}^{\infty}  e^{-\beta [ \theta(N) \, t - \mu N]}\right),\; t \in [0,1]$$
   is Dini (off course, $\psi(x):=\psi(\pi(x))$, for any $x \in \Omega$, if there is no risk of confusion).\\
   Applying our criteria and the fact that $\omega_{A_N}(t) \leq \theta(N) \, t$, we obtain
   $$\limsup_{i \to \infty} \left(\sum_{N \in \mathbb{N}_0}  \omega_{A_N}(c^{i}) \; (e^{-\beta\, \varepsilon})^{N}\right)^{1/i}  \leq \limsup_{i \to \infty} \left(\sum_{N \in \mathbb{N}_0}  \theta(N) \, c^{i}  \; (e^{-\beta\, \varepsilon})^{N}\right)^{1/i} =$$
   $$c \;\limsup_{i \to \infty} \left(\sum_{N \in \mathbb{N}_0} \theta(N) \, (e^{-\beta\, \varepsilon})^{N}\right)^{1/i} <1,$$
   provided that $\sum_{N \in \mathbb{N}_0} \theta(N) \, (e^{-\beta\, \varepsilon})^{N}< \infty$. By the root test we must have
   $$\limsup_{N \to \infty} \left(\theta(N)\right)^{1/N} <e^{\beta\, \varepsilon}$$
   to ensure that $\psi$ is Dini continuous. Particular choices would be $\theta(N):=\ln(1+N)$ or $\theta(N):=N^\alpha, \alpha>0$.  In these cases $\limsup_{N \to \infty} \left(\theta(N)\right)^{1/N} = 1 <e^{\beta\, \varepsilon}$ because $\beta\, \varepsilon>0$.
\end{example}

The next example shows that we may have a Dini continuous grand-canonical potential $\psi$ which is not Lipschitz or H\"older continuous. This shows that the generality of our results cannot be reduced to a direct application of the classical Ruelle Theorem, which is well known for the case of  Lipschitz (or H\"older) potentials (note that our IFS is contractive as in \cite{Fan}).

{\begin{example} \label{Dini-notH} We follow the notation  of Theorem 2.1 in \cite{W}.  We will present an example where $\psi$ is Dini but not Lipschitz. Consider the Bernoulli space $\Omega=\{0,1\}^\mathbb{N}$ and we are going to define a function $u:\Omega \to \mathbb{R}$.
We set  for $0<t<1$ fixed,
$$ u(0^p 1 z) = t\,  e^{p^{-2 - \epsilon}}=\gamma_p= \delta_p= u(1^p 0 z),$$
and in the other points we define $u$ in such way that $u(0z) + u(1 z)=1$, for all $z$. That is, we take
$$ u(1 0^p 1 z) =1- \gamma_{p+1}=1-  \delta_{p+1}= u(0 1^p 0 z).$$
We set $t<1$, in such that   there exist  a $c\in(0,1)$ such that $c\leq \gamma_p$ and $\gamma_p \leq (1-c)$ (see paragraph before Theorem 2.1 in \cite{W}).\\
The normalized  potential $\log u$ is Walters (also Dini)  according to Theorem 2.1 (ii).   Indeed, it follows from Lemma 2.1 in \cite{W} that the variation of $ \log u$ on a cylinder of size $p$ (in our example)  is of order $p^{-2 - \epsilon}$.  Therefore, $\psi=\log u$   is not Lipschitz (or  H\"older).

Now we are going to define a family of potentials $A_n:\Omega \to \mathbb{R}$.

For fixed $p$, and for each $N =0,1,2,...,$ define
$$B_{N,p}= B_N(0^p 1 z )=\gamma_{p,N}=\delta_{p,N}= B_N(1^p 0 z )=t\, \frac{1}{N!} [p^{-2 - \epsilon}]^N.$$
In this way $\sum_{N=0}^\infty B_N (0^p 1 z ) = t\, e^{p^{-2 - \epsilon}}= \sum_{N=0}^\infty B_N (1^p 0 z ).$
We also take
$$ B_N(1 0^p 1 z) = B_N(0 1^p 0 z),$$
in such a way that
$$ \sum_{N=0}^\infty B_N (1 0^p 1 z ) =1-  t\, e^{p^{-2 - \epsilon}}= \sum_{N=0}^\infty B_N  (0 1^p 0 z ).$$
Now taking $A_{N,p}(0^p 1 z)= - \ln B_{N,p} + N \mu $, it follows that   $e^{- [ A_N(0^p 1 z) - \mu N]}=B_{N,p}.$ In this way the normalized potential
$$\psi(x)= \ln u(x)  = \ln\left(\sum_{N \in \mathbb{N}_0}  e^{-\beta [ A_N(x) - \mu N]}\right)$$
is a Dini potential which is not Lipschitz (or H\"older).
\end{example}

Now we can state our main result in this section, the existence of a
positive eigenfunction for the grand-canonical-Ruelle operator $\mathcal{L}_{\beta, \mu}$ and an eigenprobability for $(\mathcal{L}_{\beta, \mu})^{*}$. 

\begin{theorem}
\label{thm: Fan_Boson} Consider $\mathcal{R} = (\Omega, \phi_{j}, q_{j})_{j
\in \mathcal{A}}$ the contractive system in Lemma~\ref{lem:equal_operat_finite}. If the sequence $\Phi = (A_N)_{N \in \mathbb{N}
_{0}}$, satisfies 
  \begin{equation}
\exists \,\varepsilon>0, \delta \geq 0 \text{ s.t. } \; A_N(x) > (\mu + \varepsilon) N + \delta, \; \forall x \in \Omega, \forall N \in \mathbb{N}_0,
\end{equation}
each $A_N$ is Dini continuous and, for some $0<c<1$,
$$\limsup_{i \to \infty} \left(\sum_{N \in \mathbb{N}_0}  \omega_{A_N}(c^{i}) \; (e^{-\beta\, \varepsilon})^{N}\right)^{1/i} <1,$$
then there exists a unique continuous function $h: \Omega \to \mathbb{R}$, $h >0$ and a
unique probability measure $\nu$ on $\Omega$ such that
\begin{equation*}
\mathcal{L}_{\beta, \mu} (h) = \lambda h, (\mathcal{L}_{\beta,
\mu})^{*} (\nu) = \lambda \nu  \text{ and } \nu(h)=1
\end{equation*}
where $\lambda>0$ is the spectral radius of $\mathcal{L}_{\beta, \mu}$.
Moreover, for any $g \in C(\Omega)$ the sequence of functions $\lambda^{-n} (\mathcal{L}_{\beta, \mu})^{n} (g)$ converges uniformly to $\nu(g) h$ and for any probability measure $\theta$ the sequence $\lambda^{-n} (\mathcal{L}_{\beta, \mu}^{*})^{n} (\theta)$ converges weakly to $\theta(h)\nu$.
\end{theorem}

\begin{proof}
From Lemma~\ref{lem:equal_operat_finite} a) we get $\mathcal{L}_{\beta, \mu}=B_{q}$, thus we can derive our theorem directly from Ruelle-Perron-Frobenius theorem for IFS (see \cite{Fan}, Theorem 1.1), by showing that $\log q_{j}$ is Dini continuous for any $j \in \mathcal{A}$. \\
Since $q_{j}(x):=e^{\psi(\phi_{j}(x))} > 0, \; j \in \mathcal{A}$. The above condition is equivalent to $\log q_{j} = \psi(\phi_{j}(x))$ to be Dini continuous, which is the case because $\phi_{j}$ is Lipschitz continuous and, under our hypothesis, we get from Lemma~\ref{lem:equal_operat_finite} (b) that $\psi$ is Dini continuous.

We recall that a contractive IFS always has attractors, that is, a unique compact set $K \subseteq \Omega$ such that $K=\bigcup_{j \in \mathcal{A}} \phi_{j} (K)$. Although due to the very particular structure of our maps $\phi_{j}$,  we obtain  $\Omega=\bigcup_{j \in \mathcal{A}} \phi_{j} (\Omega)$, thus we will ignore this feature when applying Fan's theorem.
\end{proof}

Now we can translate all the results from the thermodynamical formalism for
an IFS with weights $\mathcal{R}=(\Omega, \phi_{j}, q_{j})_{j \in \mathcal{A}%
}$ following \cite{LO} and its improvement given in \cite{CO17} (a different
approach is given by \cite{Mih} considering measures invariant w.r.t. a skew
product).

The next lemma generalizes the result obtained in Lemma~\ref{lemma-pressure} for the
setting of grand-canonical-Ruelle operators.

\begin{lemma}
\label{powers description} If the family of potentials $\Phi=(A_N)_{N \in \mathbb{N}_{0}} $, satisfy the hypothesis from Theorem~\ref{thm: Fan_Boson},  then
the following limit exists
\begin{align}  \label{1sobreNlnBN}
\lim_{n \to \infty} \frac{1}{n} \log\left((\mathcal{L}_{\beta, \mu})^n(1)
(x) \right) = \log \lambda \;,
\end{align}
the convergence is uniform in $x\in \Omega$ and $\lambda$ is the spectral
radius of $\mathcal{L}_{\beta, \mu}$ acting on $C(\Omega)$. We call %
\eqref{1sobreNlnBN} the log of the  \textit{grand-canonical eigenvalue} $\lambda$.
\end{lemma}

 The expression \eqref{1sobreNlnBN} also represents the dynamical partition function and is the dynamical analogous of the microcanonical partition function \eqref{hh121}.

\begin{proof}
  Consider $\mathcal{R}=(\Omega, \phi_{j}, q_{j})_{j \in \mathcal{A}}$ the contractive system in Lemma~\ref{lem:equal_operat_finite}. So, we get $\mathcal{L}_{\beta, \mu}=B_{q}$ and, thus, we can derive our theorem directly from \cite{CO17}, Lemma 1, and from Theorem~\ref{thm: Fan_Boson} who gives us a positive eigenfunction for $B_{q}$, with eigenvalue $\lambda$.
\end{proof}

\begin{definition}
\label{lolo} A \textit{holonomic probability} $\hat{\nu}$ with respect to $%
\mathcal{R}$ on $\Omega \times \mathcal{A}^\mathbb{N}$, is a probability such that
\begin{equation*}
\int_{\Omega \times \mathcal{A}^\mathbb{N}_0 } g\left(\phi_{w_1}(x)\right) d \hat{\nu}(x,w)=\int_\Omega g(x) d\nu(x),
\end{equation*}
for all $g:\Omega \rightarrow \mathbb{R}$ continuous, where $\nu$ is the
projection on the first coordinate of $\hat{\nu}$ (i.e., $\int g(x)
d\nu(x):=\int g(x) d \hat{\nu}(x,w)$). The set of all holonomic probability
measures with respect to $\mathcal{R}$ is denoted $\mathcal{H}(\mathcal{R})$.
\end{definition}

\begin{definition}[\protect\cite{LO,CO17}]
\label{entropy} Let $\mathcal{R}=(\Omega, \phi_{j})_{j \in \mathcal{A}}$ be
an IFS and $\hat{\nu} \in \mathcal{H}(\mathcal{R})$. The \textit{variational
entropy} of $\hat{\nu}$ is defined by
\begin{equation}  \label{fie}
h_v(\hat{\nu}) \equiv \inf_{g\in \mathrm{C}^+(\Omega)} \left\{ \int_{\Omega}
\log \frac{B_{1}(g)(x)}{ g(x) } d\nu(x) \right\}
\end{equation}
where $B_{1}(g)(x)= \sum_{j \in \mathcal{A}} g(\phi_{j}(x))$.

Call such  variational entropy  the grand-canonical entropy when applied to the case we consider here.

\end{definition}

From \cite{LO}, Proposition 19 (see also \cite{CO17}, Theorem 10), we know
that $0 \leq h_v(\hat{\nu}) \leq r=\sharp(\mathcal{A})$ for any $\hat{\nu}
\in \mathcal{H}(\mathcal{R})$.

Inspired by \cite{CO17}, Definition 11, we can now introduce the concept of
grand-canonical topological pressure.

\begin{definition}
\label{Pressure} Consider the IFS with weights $\mathcal{R}=(\Omega,
\phi_{j}, q_{j})_{j \in \mathcal{A}}$. The \textit{grand-canonical
topological pressure} of $Q=(q_{j})_{j \in \mathcal{A}}$, is defined by the
following expression
\begin{equation}  \label{fip}
P(Q) :=\sup_{\hat{\nu} \in \mathcal{H}(\mathcal{R})} \inf_{ g \in
C^+(\Omega)} \left\{ \int_{\Omega} \log \frac{\mathcal{L}_{\beta, \mu}(g)}{
g } d\nu \right\},
\end{equation}
assuming $\mathcal{L}_{\beta, \mu} = B_{q}$.
\end{definition}

\begin{theorem}
\label{thm: equili_IFS_fin} Consider $(\Omega, \phi_{j}, q_{j})_{j \in
\mathcal{A}}$ the contractive system in Lemma~\ref{lem:equal_operat_finite}.
Assume that  the  family of potentials  $\Phi=(A_N)_{N \in \mathbb{N}_{0}}$, satisfy the hypothesis from Theorem~\ref{thm: Fan_Boson}  and $\mathcal{L}_{\beta, \mu} = B_{q}$.

Denote by $\psi(y):= \ln\left(\sum_{N \in
\mathbb{N}_0} e^{-\beta [ A_N(y) - \mu N]}\right),\; y \in \Omega$, the grand-canonical potential. Then:

\begin{enumerate}
\item[a)] $\displaystyle P(Q)=P(\psi):= \sup_{\hat{\nu} \in \mathcal{H}(
\mathcal{R})} \left\{ h_{v}(\hat{\nu})+ \int_{\Omega} \psi(y)\, d\nu(y)\right\}.$\\

\item[b)] The set of equilibrium states, that is, holonomic measures $\hat{%
\mu}$ satisfying $P(\psi) = h_{v}(\hat{\mu})+ \int_{\Omega} \psi\,
d\mu$, is not empty;

\item[c)] $P(Q)=\log(\lambda)$, where $\lambda$ is the  grand-canonical eigenvalue given by
Theorem~\ref{thm: Fan_Boson}.
\end{enumerate}
\end{theorem}

\begin{proof}
a) In the general setting presented in \cite{CO17}, Definition 11, one can consider the homogeneous case, that is, when $\psi:\Omega \to \mathbb{R}$ is a positive continuous function  and $\mathcal{R}= (\Omega, \phi_{j}, q_{j})_{j \in \mathcal{A}}$, with $q_{j}=\psi\circ\phi_{j}$ for each $j \in \mathcal{A}$. In this case, the topological pressure of $\psi$ is alternatively given by
\[
P(\psi)=
\sup_{\hat{\nu} \in \mathcal{H}(\mathcal{R})}
\left\{ h_{v}(\hat{\nu})+  \int_{\Omega} \psi\,  d\nu\right\}.
\]
Actually, this is the case for the weights $q_j$ obtained for a family $\Phi=(A_N)_{N \in \mathbb{N}_{0}}$ as in Lemma~\ref{lem:equal_operat_finite}. Indeed, given $\displaystyle q_{j}(x):= \sum_{N \in \mathbb{N}_0}  e^{-\beta [ A_N(jx) - \mu N]} > 0,\;$ $ j \in \mathcal{A}$ we can choose the function
\begin{equation} \label{liy}\psi(y):=\ln\left( \sum_{N \in \mathbb{N}_0}  e^{-\beta [ A_N(y) - \mu N]}\right),\end{equation}
which is well defined. It follows immediately that $q_{j}=e^{(\psi\circ\phi_{j})}, \; j \in \mathcal{A} $.
b) We say that the holonomic measure $\hat{\mu}$ is an equilibrium state for
$\psi$ if $h_{v}(\hat{\mu})+  \int_{\Omega} \psi(x)\, d\mu(x) = P(\psi).$
From \cite{CO17}, Theorem 13,  we know that for such IFS and $\psi:\Omega\to\mathbb{R}$ a continuous function, so that $e^\psi$ is positive, the set of equilibrium states for $\psi$ is not empty.\\
c) Is a direct consequence of \cite{LO}, Theorem 22, and Theorem~\ref{thm: Fan_Boson}.
\end{proof}

\medskip

The next corollary highlights the importance of the results about the existence of an eigenfunction in our version of  Ruelle's Theorem.  Under the validity of the  hypothesis of Theorem~\ref{thm: Fan_Boson} we will show the analytic dependence  of the eigenfunction $\lambda$ as a function of $\psi$ in \eqref{liy}.

\smallskip

\begin{corollary} \label{ana}
  If the family $\Phi=(A_N)_{N \in \mathbb{N}_{0}}$, satisfies the hypothesis of Theorem~\ref{thm: Fan_Boson} and the family of potentials is uniformly Lipschitz, that is,  $\operatorname{Lip}(A_N) \leq M$, for some fixed $m$, then, the eigenfunction  varies analytically  on the grand-canonical potential $\psi$ given by \eqref{liy}. From this follows the analyticity of the pressure (and also of the eigenvalue $\lambda$) as a function of $\psi$.
\end{corollary}
\begin{proof} First note that from our hypothesis we get, from Lemma~\ref{lem:equal_operat_finite} (b), that the weights of our IFS are Lipschitz continuous. 

As seen in \cite[Remark 2]{BCLMS23}, whose reasoning is similar to \cite{MR1085356}, the positiveness of the transfer operator (and the fact that the attractor is $\Omega$) means that the dimension of the eigenspace associated to the spectral radius $\lambda$ is one. Also, since the weights are Lipschitz continuous according to Lemma~\ref{lem:equal_operat_finite}, we know from \cite{MR1085356} that the essential spectrum of the operator is contained in a disc of radius strictly smaller than $\lambda$. For a general version of this result (which contemplates the case of the Grand-canonical potential $\psi$ considered by us) we refer the reader to \cite{He} and \cite{Ye}. This implies that the main eigenvalue is isolated. Using a standard argument of complex analysis (Cauchy's integral formula for bounded operators on Banach spaces) we get the analyticity (see for instance Theorem 5.1 in \cite{Mane}, \cite{SilvaFe},
 or Proposition 35 and 36 in \cite{Lop1}); the reasoning here  should follow exactly the same procedures: take a circle path around the eigenvalue on the complex plane, etc,  and we leave the  details for the reader.
\end{proof}


\subsection{An alternative setting via an infinite countable IFS with
weights.}

\label{outy}

The idea in this section is to choose an infinite countable IFS with weights
whose transfer operator matches the grand-canonical-Ruelle operator $%
\mathcal{L}_{\beta, \mu}$ (that appears in \eqref{BRO}), so that the
thermodynamical formalism for $\mathcal{L}_{\beta, \mu}$ can be derived from
the well-known thermodynamical formalism for that kind of system, for
details see \cite{Urb}. We have to show that our model fits the hypothesis
of \cite{Urb}.

In this section, the emphasis will be on results for the grand-canonical-dual-Ruelle operator $\mathcal{L}_{\beta, \mu}^*$ and not so much for the
grand-canonical-Ruelle operator $\mathcal{L}_{\beta, \mu}$. We will be able
to obtain the same claims as presented in Theorem \ref{thm: equili_IFS_fin}
but without the need to show the existence of an eigenfunction for the
operator $\mathcal{L}_{\beta, \mu}$.

We start by setting up the appropriate IFS and choosing maps and weights.
After showing that the transfer operator for that IFS coincides with $%
\mathcal{L}_{\beta, \mu}$ we prove that the IFS satisfies the regularity
requirements from \cite{Urb}. Finally, we introduce results on partition
functions and topological pressure characterizing the topological pressure
through the eigenvalue of the dual operator $\mathcal{L}_{\beta, \mu}^{*}$.
In this framework, we do not provide a version of the Ruelle theorem for $%
\mathcal{L}_{\beta, \mu}$ because it is not  known of  any result about the
existence of eigenfunctions in the IFS literature, to the best of our
knowledge.

To simplify the notation, we consider now the alphabet $\mathcal{A}%
:=\{0,...,r-1\}$ instead of $\mathcal{A}:=\{1,...,r\}$. In particular, $%
\Omega:=\mathcal{A}^{\mathbb{N}}=\{0,...,r-1\}^{\mathbb{N}}$ is our symbolic
metric space endowed with the metric
\begin{equation*}
d(x, y) := \left\{
\begin{array}{ll}
2^{-\min\{n \in \mathbb{N} : x_n \neq y_n \}}, & x \neq y; \\
0, & x=y.%
\end{array}
\right.
\end{equation*}
Consider the \textit{countable IFS} $\mathcal{R}:=(\Omega, \phi_{j})_{j \in
I}$ on $I := \mathbb{N}_{0}$, where $\phi_{j}(x)=(j \mod{r})x$ is the
\textit{mnemonic representation} for $(j \mod r, x_1, x_2, ...)$ and $%
x=(x_1, x_2, ...) \in \Omega$. This IFS is obviously contractive w.r.t. the
distance introduced in $\Omega$, that is, ${\rm Lip} (\phi_{j}) = \frac{1}{2} $
for all $j \in \mathbb{N}_0$. Note that the sequence of maps is formed by
repetitions, $\phi_{0}(x)=0x$, $\phi_{1}(x)=1x$, ..., $\phi_{r-1}(x)=(r-1)x$%
, $\phi_{r}(x)=(r \mod r) \, x=0x$, $\phi_{r+1}(x)=(r+1 \mod r)x=1x$, $%
\phi_{r+2}(x)=(r+2 \mod r)x=2x$, and so on.

Here we follow \cite{Urb}, where only positive weights are allowed in  the Ruelle operator. This is
achieved by considering  a countable  \textit{IFS with weights} $\mathcal{R}%
:=(\Omega, \phi_{j}, q_{j})_{j \in \mathbb{N}_0}$, where $Q:=\{ q_{j}: \Omega
\to \mathbb{R}, \; j \in \mathbb{N}_0 \}$ is family of continuous functions.
The new \textit{transfer operator} $B_{q}: \mathrm{C}(\Omega) \to \mathrm{C}%
(\Omega)$ is given by
\begin{equation*}
B_{q}(g)(x)= \sum_{j \in \mathbb{N}_0} e^{q_{j}(x)} g(\phi_{j}(x)),
\end{equation*}
for any $g \in \mathrm{C}(\Omega)$.


\begin{lemma}
\label{weight_count_family} Suppose that $\Phi=(A_N)_{N \in \mathbb{N}_0}$
is an admissible sequence of potentials satisfying  ${\rm Lip} (A_{N}) \leq M$, for some $M>0$. Consider
the function $\xi: \mathbb{N}_0 \to \mathbb{N}_0$ defined by
\begin{equation*}
\xi(j)=\frac{j - (j\mod r)}{r}.
\end{equation*}
If we choose the weights $q_{j}(x):= - \,\beta\, (A_{\xi(j)}(\phi_{j}(x))-
\xi(j)\, \mu), \; j \in \mathbb{N}_0$ and construct the IFS with weights $%
\mathcal{R}:=(\Omega, \phi_{j}, q_{j})_{j \in \mathbb{N}_0}$ then
\begin{equation*}
B_{q}(g)(x)= \mathcal{L}_{\beta, \mu}(g) (x)
\end{equation*}
for any $g \in \mathrm{C}(\Omega)$. In particular $(q_{j})_{j \in \mathbb{N}%
_0}$ is a sequence of uniformly Lipschitz continuous functions with $\sup_{j
\in \mathbb{N}_{0}} {\rm Lip} (q_{j}) \leq \frac{\beta\, M}{2}.$
\end{lemma}

\begin{proof}
   The proof follows easily from a computation.  In order to obtain a representation of the operator $\mathcal{L}_{\beta, \mu}$ as the transfer  operator of an IFS we must find a suitable family of weights. To do that we observe that
$$\mathcal{L}_{\beta, \mu}(f) (x)
:=\sum_{N \in \mathbb{N}_0} e^{\beta \, \mu\,  N }\, \sum_{j \in \mathcal{A}} e^{- \,\beta\, A_N(jx)} f(jx)= $$
$$= e^{\beta \, \mu\,  0 }[e^{- \,\beta\, A_0(0x)} f(0x)+ e^{- \,\beta\, A_0(1x)} f(1x)+ \cdots + e^{- \,\beta\, A_0((r-1)x)} f((r-1)x)] + $$
$$+ e^{\beta \, \mu\,  1 }[e^{- \,\beta\, A_1(0x)} f(0x)+ e^{- \,\beta\, A_1(1x)} f(1x)+ \cdots + e^{- \,\beta\, A_1((r-1)x)} f((r-1)x)] + \cdots =$$
$$= e^{- \,\beta\, (A_0(0x)- \mu\,  0)} f(0x)+ e^{- \,\beta\, (A_0(1x)- \mu\,  0)} f(1x)+ \cdots + e^{- \,\beta\, (A_0((r-1)x)- \mu\,  0)} f((r-1)x) + $$
$$+ e^{- \,\beta\, (A_1(0x)- \mu\,  1)} f(0x)+ e^{- \,\beta\, (A_1(1x)- \mu\,  1)} f(1x)+ \cdots + e^{- \,\beta\, (A_1((r-1)x)- \mu\,  1)} f((r-1)x) + \cdots.$$
In the first line we replace $0x= \phi_{0}(x)$, $1x= \phi_{1}(x)$, ..., $(r-1)x= \phi_{(r-1)}(x)$ and after, in the second line, we replace $0x= \phi_{r}(x)$, $1x= \phi_{r+1}(x)$, ..., $(r-1)x= \phi_{(2r-1)}(x)$, and so on.

From this choice for the coefficients we get:
$$q_{0}(x):= - \,\beta\, (A_0(0x)- \mu\,  0), \; q_{1}(x):= - \,\beta\, (A_0(1x)- \mu\,  0), ...,$$$$ \;q_{r-1}(x):= - \,\beta\, (A_0((r-1)x)- \mu\,  0),$$
$$q_{r}(x):= - \,\beta\, (A_1(0x)- \mu\,  1), \; q_{r+1}(x):= - \,\beta\, (A_1(1x)- \mu\,  1), ...,$$$$ \; q_{2r-1}(x):=- \,\beta\, (A_1((r-1)x)- \mu\,  1),$$
$$q_{2r}(x):= - \,\beta\, (A_2(0x)- \mu\,  2), \; q_{2r+1}(x):= - \,\beta\, (A_2(1x)- \mu\,  2), ...,$$ $$ \; q_{3r-1}(x):=- \,\beta\, (A_2((r-1)x)- \mu\,  2),...$$
which obviously satisfy
$q_{j}(x):= - \,\beta\, (A_{\xi(j)}(\phi_{j}(x))- \xi(j)\, \mu), \; j \in \mathbb{N}_0$, proving our claim.
\end{proof}

Despite the fact that the maps $\phi_j$, $j \in \mathbb{N}$,   repeat themselves periodically, the weights do not. Thus, $\mathcal{R}= (\Omega, \phi_j, q_j)$ is a genuine countable IFS with weights.

To further developments we need to introduce some notation on countable IFSs taken
from~\cite{Urb}.

A word of length $n$ in $I^{\mathbb{N}}$ is an element $w:=(w_{1},
...,w_{n}) $ of $I_{n}:=I^{n}$ and $\sigma(w):=(w_{2}, ...,w_{n}) \in I_{n-1}
$. Given $x \in \Omega$ the iterate $\phi_{w}(x)$ is the point
\begin{equation*}
y := \phi_{w_{1}}(\cdots (\phi_{w_{n}}(x)) \in \Omega.
\end{equation*}

A family of continuous functions $Q:=\{ q_{j}: \Omega \to \mathbb{R}, \; j
\in I \}$ is $\alpha$-H\"older if
\begin{equation*}
V_{\alpha}(Q):=\sup_{n\geq 1} V_{n}(Q) <\infty
\end{equation*}
where
\begin{equation*}
V_{n}(Q):=\sup_{w \in I_{n}} \sup_{x\neq y} |q_{w_{1}}(\phi_{\sigma(w)}(x))
- q_{w_{1}}(\phi_{\sigma(w)}(y))| e^{\alpha (n-1)}.
\end{equation*}
Additionally, if $B_{q}(1) \in C(\Omega)$ then we say that the family $Q$ is
strongly $\alpha$-H\"older or, according to \cite{Urb0}, \emph{summable}.

\begin{lemma}
Suppose that $\Phi=(A_N)_{N \in \mathbb{N}_0}$ is a  family of potentials
satisfying  ${\rm Lip} (A_{N}) \leq M$, for some $M>0$ and
\begin{equation}\label{eq: asd1}
 A_N(x) > \left(\mu + \frac{\log(r)}{\beta}\right)\, N, \; \forall x \in \Omega, \forall N \in \mathbb{N}_0.
\end{equation}

Then, the family $Q$ is $\alpha$-H\"older for $\alpha:= \log (2)$ and
summable, that is $\mathcal{L}_{\beta, \mu}(\mathit{1}) \in C(\Omega)$.
\end{lemma}

\begin{proof}
   Consider $w \in I_{n}$ then
   $$|q_{w_{1}}(\phi_{\sigma(w)}(x)) - q_{w_{1}}(\phi_{\sigma(w)}(y))| \leq \operatorname{Lip}(q_{w_{1}}) \frac{1}{2^{n-1}}\, d(x,y) \leq $$ $$\leq\frac{1}{2^{n-1}} \frac{\beta\, M}{2}\, d(x,y)= \frac{\beta\, M}{2^{n}}\, d(x,y),$$
   because we get from Lemma~\ref{weight_count_family} that $\displaystyle\sup_{j \in \mathbb{N}_{0}} \operatorname{Lip}(q_{j}) \leq \frac{\beta\, M}{2}.$ Thus,
   $$V_{n}(Q) \leq \sup_{w \in I_{n}} \sup_{x\neq y} \frac{\beta\, M}{2^{n}} \, d(x,y)  e^{\alpha (n-1)} \leq $$ $$\leq \beta\, \frac{M}{2} \operatorname{diam}(\Omega) \frac{e^{\alpha (n-1)}}{e^{\log(2) (n-1)}} \leq \beta\, \frac{M}{2} \operatorname{diam}(\Omega),$$
   for $\alpha=\log(2)$.

   To see the second part we recall that from Lemma~\ref{weight_count_family} we get
   $$ \mathcal{L}_{\beta, \mu}({\it 1}) (x)= B_{q}({\it 1})(x)=\sum_{N \in \mathbb{N}_0}\, \sum_{j \in \mathcal{A}} e^{-\,\beta \, (\,A_N(jx) - \mu\,  N \,) }.$$
   Thus $\mathcal{L}_{\beta, \mu}({\it 1})  \in C(\Omega)$ if and only if the {\bf positive series  above} is convergent. The root test claims that it is sufficient to prove that
   $$\limsup_{N \to \infty}  \sqrt[N]{\sum_{j \in \mathcal{A}}e^{-\,\beta \, (\,A_N(jx) - \mu\,  N \,) }} <1.$$
   Recall that for nonnegative real  numbers $t_1,...., t_k$ we have $\sqrt[n]{\sum_{i=1}^{k} t_i} \leq \sum_{i=1}^{k}\sqrt[n]{ t_i}$, thus
   $$\limsup_{N \to \infty}  \sqrt[N]{\sum_{j \in \mathcal{A}}e^{-\,\beta \, (\,A_N(jx) - \mu\,  N \,) }} \leq \limsup_{N \to \infty}\sum_{j \in \mathcal{A}}\sqrt[N]{e^{-\,\beta \, (\,A_N(jx) - \mu\,  N \,) }} \leq$$
   $$\leq \sum_{j \in \mathcal{A}}e^{\limsup_{N \to \infty} \,( -\,\beta ) \left(\,\frac{1}{N}A_N(jx) - \mu \right) } < \sum_{j \in \mathcal{A}}e^{\limsup_{N \to \infty} \, ( -\,\beta ) \,\left([\frac{\log(r)}{\beta} + \mu]\, - \mu \right) }= 1,$$
   because $\displaystyle A_{N}(y) > N\, \left(\frac{\log(r)}{\beta} + \mu\right)$, for any $y \in \Omega$.
\end{proof}

It follows  from the above  that we can apply the results from \cite{Urb} regarding entropy
and topological pressure for IFS.

\begin{lemma} \label{opry}
Suppose that $\Phi=(A_N)_{N \in \mathbb{N}_{0}}$ is a family of potentials
satisfying ${\rm Lip} (A_{N}) \leq M$, for some $M>0$ and
\begin{equation*}
 A_N(x) > \left(\mu + \frac{\log(r)}{\beta}\right)\, N, \; \forall x \in \Omega, \forall N \in \mathbb{N}_0.
\end{equation*}
  The number $\lambda := \mathcal{L}_{\beta, \mu}^*(\nu_q)(\mathit{1})$ is
an eigenvalue of the dual operator $\mathcal{L}_{\beta, \mu}^{*}$ associated
to an eigenmeasure $\nu_{q}$, that is $\mathcal{L}_{\beta, \mu}^{*}(\nu_{q})
= \lambda \nu_{q}$.
\end{lemma}

\begin{proof}
   As $\mathcal{L}_{\beta, \mu}=B_{q}$ from Lemma~\ref{weight_count_family}, the result goes as follows. First we notice that $B_{q}^{*}: C(\Omega)^{*} \to C(\Omega)^{*}$ is well defined because the family $Q$ is summable ($B_{q}^{*}(\nu)({\it 1}) < \infty$) and the operator  $\nu \to \frac{B_{q}^{*}(\nu)}{B_{q}^{*}(\nu)({\it 1})}$ has a fixed point by Schauder-Tychonoff's Theorem (see \cite{Urb} for details). Lets say that the probability measure $\nu_{q}$ is that fixed point, then $\frac{B_{q}^{*}(\nu_{q})}{B_{q}^{*}(\nu_{q})({\it 1})}= \nu_{q}$ or equivalently $B_{q}^{*}(\nu_{q}) = \lambda \nu_{q}$, where $\lambda:=B_{q}^{*}(\nu_{q})({\it 1}) \in \mathbb{R}$.
\end{proof}

\begin{remark}
As we pointed out before a contractive countable IFS has an attractor $K
\subseteq \Omega$ satisfying $K=\bigcup_{j \in I}(K)$ which is not
necessarily compact, unless $I$ is a finite set. However in our case $%
\phi_{j}(x)=(j \mod r)x,\; j \in \mathbb{N}_0$ is actually a finite family
and $K =\Omega$ which is compact, by construction. Thus Lemma 2.5 from
\cite{Urb}, claiming that $\nu_{q}(K)=1$, says only that $\nu_{q}$ has full
support in $\Omega$.
\end{remark}

Then, we define the partition function associated with the family $Q$,
previously defined
\begin{equation*}
Z_{n}(Q):=\sum_{w \in I_{n}}\|e^{\sum_{j=1}^{n} q_{w_{j}}(\phi_{\sigma^{j}(w)})
}\|_{0}=\sum_{w \in I_{n}} e^{\sup_{x \in \Omega}\sum_{j=1}^{n}
q_{w_{j}}(\phi_{\sigma^{j}(w)}(x))}.
\end{equation*}

Since the function, $\log(Z_{n}(Q))$ is subadditive the topological pressure
of a $\alpha$-H\"{o}lder summable family $Q$ can be defined analogously to
the classical theory
\begin{equation}  \label{ainf}
P(Q):=\lim_{n\to \infty} \frac{1}{n} \log(Z_{n}(Q))=\inf_{n \geq 1} \frac{1}{%
n} \log(Z_{n}(Q)).
\end{equation}

A useful result from \cite{Urb} is the following one.

\begin{proposition}[\protect\cite{Urb}, Proposition 2.3]
The function $Q \to P(Q)$ is lower semicontinuous on the space of all $%
\alpha $-H\"{o}lder summable families w.r.t. the topology of the uniform
convergence.
\end{proposition}

We also have the fundamental characterization of the eigenvalue $\lambda=(%
\mathcal{L}_{\beta, \mu})^{*}(\nu_{q})(\mathit{1})$ in terms of the pressure:

\begin{lemma}
Suppose that $\Phi=(A_N)_{N \in \mathbb{N}_0}$  is a family of potentials
satisfying ${\rm Lip} (A_{N}) \leq M$, for some $M>0$ and
\begin{equation*}
 A_N(x) > \left(\mu + \frac{\log(r)}{\beta}\right)\, N, \; \forall x \in \Omega, \forall N \in \mathbb{N}_0.
\end{equation*}
 The eigenvalue $\lambda$ of the dual
operator $(\mathcal{L}_{\beta, \mu})^{*}$ is given by
\begin{equation*}
\lambda=e^{P(Q)}.
\end{equation*}
\end{lemma}

\begin{proof}
  From \cite{Urb}, Lemma 2.4, the eigenvalue $\lambda$ of the dual operator $B_{q}^{*}$ is given by  $\lambda=e^{P(Q)}.$ As $B_{q} =\mathcal{L}_{\beta, \mu}$, from Lemma~\ref{opry}, the result follows.

  \end{proof}

\section{Appendix - A brief account on Thermodynamics of ideal gases}

\label{Tgas}

This section does not primarily have a dynamical system content; our goal is
to present a brief description of concepts and phenomena occurring in the
physical world, aimed at an audience of readers who are mathematically
oriented. Here, we are interested in physical systems which are in
thermodynamical equilibrium. For simplicity, we mainly consider the case of
ideal gases (i.e., systems of non-interacting point-like particles).
Thermodynamics is the branch of physics that organizes systematically the
empirical laws referring to the thermal behavior of the macroscopic world.
It is one of our intentions to explain the meaning of this statement. For
the mathematical reader who tries to understand the content of some texts in
Physics, an initial difficulty is the jargon used there; all this will be
translated here into a more formal context. At the end of this Appendix, we
will mention a possible interpretation of topological pressure as related to
gas pressure (see Remark \ref{rel}).

Our goal is to draw a parallel between concepts of thermodynamics of gases
with similar ones in the mathematical theory of Thermodynamic Formalism in
the sense of \cite{MR1085356}. We believe that this can provide an
enrichment of the class of questions that can be raised and proposed in
Thermodynamic Formalism.

\subsection{The case of a definite number of particles}

Consider a classical gas with $N$ particles (initially $N$ is fixed) at
\textit{temperature} $T$ in a region with \textit{volume} $V$.  We
assume that the temperature  $T > 0$ is related to $\beta$ through equation \eqref{Boltzmann-constant}. Denote by $p$ the
\textit{gas pressure} of the system, by $S$ the \textit{entropy}, and by $%
\mu $ the \textit{chemical potential}. The \textit{total energy} $U$ is a
function $U = U (S, V, N)$, of the variables $S,V,N$. It is important to
point out that we are assuming that the macroscopic variables $T,V,N,p$ can
be measured. In fact, the values of these variables $S,V,N$, are not so
important in themselves, but their variations $\delta S,\delta V,\delta N$
are.

For the benefit of the reader, we will briefly describe below some basic
properties of the thermodynamics of gases (intertwined with Statistical
Mechanics). For more details see \cite{MR2777415}, \cite{Salinas}, \cite{Sch}%
, \cite{Cale}, \cite{Zu} or Section 6 in \cite{LR}. In this way, we think
that some of the future (and also past) definitions that we will present
here will look natural.


The fundamental relations concerning the above-described variables can be
found in (3.6) on page 42 on \cite{Salinas}:

\begin{align*}
T &= \frac{\partial U}{\partial S}; \\
p &=- \frac{\partial U}{\partial V}; \\
\mu &= \frac{\partial U}{\partial N} \;.
\end{align*}

Another fundamental relation in thermal physics (of ideal gases) is the
following

\begin{equation}
p\,V=k_{B}\,N\,T\;.  \label{fun1}
\end{equation}%
Recall that $k_{B}$ denotes the Boltzmann constant. The above equation is
called the equation of state of the ideal gas. Other kinds of systems (like
interacting gases, liquids, solids, etc.) have their own equations of state.

The above can be understood in several different ways from a physical point
of view. For instance, consider a piston at temperature $T$, where the
piston chamber has volume $V$ and contains $N$ particles. When the volume $V$ and the temperature $T$ are fixed, the
pressure $p$ becomes a linear function of the number of particles $N$. On
the other hand, when the volume $V$ and the number of particles $N$ are
fixed, it follows that the pressure $p$ is linear with respect to the
temperature $T$. Later, we will be interested in the case where the number $N
$ of particles is an unknown value ranging on $\mathbb{N}_{0}$ (the grand
canonical case).

We point out that the relation \eqref{fun1} is valid under the quasi-static
regime (this means that the thermodynamic processes we consider are such
that the changes are slow enough for the system to remain in equilibrium).

A variable is called \textit{extensive (intensive)}, if it is proportional
(not proportional) to the volume $V$, like energy and number of
particles (energy and particle densities). More precisely,
it is important to recall that the definitions of extensive and intensive
variables make sense only in the thermodynamic limit, i.e., extensive
(intensive) variables are proportional (not proportional) to the volume $V
$, when $V$ tends to infinity. For instance, if the ratio $N/V$ tends
to some constant $\varrho $, then the particle density, which is
precisely $\varrho $, is an intensive quantity, whereas the number
of particles $N$ is the associated extensive quantity. In a similar way as
for intensive quantities, one also defines so-called molar quantities, which
are, by definition, quantities that are proportional to the number of
particles $N$. For instance, the total energy per particle is a molar
quantity if the ratio $E/N$ has a limit, as $N\rightarrow \infty $.
Frequently, the ratio $N/V$ is set to be constant, that is, the particle
density is fixed. In this particular situation, molar and intensive
quantities are, of course, equivalent notions.

We point out that the mathematical formalism for Gas Thermodynamics and
Statistical Mechanics is not exactly the same, but certain general
principles are common in both theories. Particles of a gas are displayed in
a random way, as well as spins on a lattice. The randomness of the particles
of gas (for instance each position and velocity) should be described by a
probability. Note that a probability (a law that is assigned to Borel sets
values in $[0,1]$) is not a physical entity; it is a tool to predict (or to
explain) - values that are measured in the Physics of the real world - in
circumstances of lack of complete knowledge (for a discussion relating
Physics to Information Theory see for instance \cite{Cat} and \cite{Bricm}).

For practical purposes (aiming the reader familiar with probability) one can
identify \textit{microstates} with points in $\{1,2,..,d\}^\mathbb{N}$,
\textit{ensembles} with probabilities in $\{1,2,..,d\}^\mathbb{N}$, and
finally \textit{macrostates} with continuous functions $A: \{1,2,..,d\}^%
\mathbb{N} \to \mathbb{R}$. Given a probability (an ensemble) $\rho$ in $%
\{1,2,..,d\}^\mathbb{N}$, the value
\begin{equation*}
<A>_\rho=\int A d \rho,
\end{equation*}
is considered a \textit{macroscopical quantity.} Macroscopic variables are
easier to measure.

What are the probabilities (ensembles) $\rho$ which are relevant when an
isolated system is governed by a certain Hamiltonian (a macrostate) $A$ and
it is at temperature $T=\frac{1}{k_B\,\beta}$? Typically one is interested
in a minimization (or maximization) problem-related to free energy - where
an illustration of this problem is given in \eqref{preu} - (or the MaxEnt
method, where an exemplification is given in \eqref{maxu}), and this
requires the addition of the concept of entropy (to be introduced soon in
the comments to the Second Postulate of Thermodynamics). Equilibrium in
isolated systems occur according to the principle of entropy maximization
(the Second Law of Thermodynamics).

Related to the above-mentioned variational problems, it is worthwhile to
consider the concept of canonical distribution. Given $\beta>0$, a
macrostate $A$, and an \textit{a priori} measure $\rho$ on $%
\{1,2,..,d\}^\mathbb{N}$, the \textit{microcanonical partition function} is
\begin{equation}  \label{hh121}
Z (\beta) =\int e^{- \beta A(x)} d \rho (x)<\infty,
\end{equation}
and the \textit{canonical distribution} $\mu_{A,\rho, \beta}$ is given by
the law
\begin{equation}  \label{hh1}
B \to \mu_{A,\rho, \beta}(B)= \frac{\int_B e^{- \beta A(x)} d \rho (x)}{
Z(\beta)}.
\end{equation}

We call $\mu_{A,\rho, \beta}$ the \textit{microcanonical distribution} (or
else, \textit{microcanonical ensemble}) for $A, \beta$ and the \textit{a
priori} measure $\rho$.

The importance of this class of probabilities is due to the fact that they
are the solutions to certain kinds of variational problems (see Remark \ref%
{hhy}) which are related to the Second Law of Thermodynamics. Therefore, $%
\mu_{A,\rho, \beta}$ describes an equilibrium ensemble (state).

When $\rho$ is fixed, all these probabilities $\mu_{A,\rho, \beta}$ are
absolutely continuous with respect to each other.

When $\rho$ is the counting measure on $\{1,2,...,d\}$ and $A:\{1,2,...,d\}
\to \mathbb{R}$, the \textit{microcanonical partition function} is
\begin{equation*}
Z (\beta) = \sum_j e^{- \beta A(j)},
\end{equation*}
and the \textit{canonical distribution} $\mu=\mu_{A, \beta}$ is the
probability such that
\begin{equation}  \label{hh1}
\int f d \mu_{A,\beta}= \frac{\sum_j f(j) e^{- \beta A(j)} }{ Z(\beta)}.
\end{equation}

The $\mu_{A, \beta}$-probability of $j_0$ is
\begin{equation}  \label{hu1}
\frac{ e^{- \beta A(j_0)} }{ \sum_j e^{- \beta A(j)}}.
\end{equation}

\smallskip

All the above does not have a \textit{dynamical} content. Entering in a
dynamically context we can say that an \textit{equilibrium ensemble} can be
seen as a shift-invariant probability on $\{1,2,..,d\}^\mathbb{N}$. In
Statistical Mechanics the dynamics of the shift describe translation on the
one-dimensional lattice.

Note that ergodic probabilities are singular with respect to each other;
therefore, the dynamical point of view presents some conceptual differences
when compared with the above.

\bigskip

Energy, volume and the number of particles $N$ are called the \textit{%
macroscopic extensive parameters}. The \textit{microscopical variables}
refer to probabilities (states), and the main issue (see Remark \ref{dance})
is to establish a connection between them and the visible variables of the
macroscopic world (that is, with thermodynamics).

We observe, from a historical perspective, that initially the thermodynamics
of gases was developed without the knowledge that gas was constituted by a
large number of particles.

Now we will describe the \textit{postulates of equilibrium thermodynamics}
in a simplified way (for more details see pages 40-41 in \cite{Salinas} or
Section 14.4 in \cite{Bai}). These laws describe the interplay between
microscopical and macroscopical variables. They are rules of Nature, and
their validity derives from the fact that when assuming them, the
consequences that are inferred are in accordance with the observed reality.

\smallskip

\textbf{First Postulate:} The macroscopic state is completely characterized
by the internal energy $U$, the volume $V$, and the number of particles $N$.
The total energy of an isolated system (for which energy and matter transfer
through the system boundary are not possible) is conserved. \smallskip

\textit{Comment: In Thermodynamic Formalism, the potential $A=-\,H$, where $%
H $ plays the role of the Hamiltonian and corresponds to the concept of
internal energy. From \cite{Cale}: ``Energy $U$ is transferred between
systems in two forms: energy in the form of work $W$, and disordered energy
in the form of heat $Q$... Note that heat $Q$ and work $W$ are not
themselves functions of the state of a system; they are measures of the
amount of energy $\delta U=\delta W + \delta Q$ that is transferred between
systems in different forms.'' }

\smallskip

Relations between heat $Q$, work $W$ and energy $U$ in a dynamical setting
are described in Section 7 on \cite{LR}.

\textbf{Second Postulate:} There is a function of the extensive parameters $%
U,V,N$ called entropy, denoted by $S = S(U,V,N)$, that is defined for all
states of equilibrium. If we remove an internal constraint, in the new state
of equilibrium the extensive parameters of the system assume a set of values
that maximize the entropy. The entropy, as a function of the extensive
parameters, is a fundamental equation of the system. It contains all the
thermodynamic information about this system. For isolated systems, entropy
never decreases.

\smallskip

\textit{Comment: When considering probabilities $p=(p_1,p_2,...,p_d)$ on the
set $\{1,2,...,d\}$, the entropy $h(p) =-\sum_{j=1}^d \log p_j\, p_j \geq 0.$
The above postulate is one of the forms of the second law of thermodynamics
(see \cite{Sch} or Section 6 in \cite{LR} either), a topic which is
discussed in the dynamical setting in Sections 3 and 4 in \cite{LR}. This is
related to the principle of maximization of entropy when the mean energy (or
other variables) is fixed (see \eqref{maxu} for a particular case).
Expression (87) in Section 7 in \cite{LR} illustrates the MaxEnt principle
in Thermodynamic Formalism. Entropy plays a fundamental role in the search
for equilibrium via the variational problem associated with the Second Law.
This makes the theoretical model in consonance with what is physically
observed. }

\textit{The probability with maximal entropy can be seen as the one that
contains the maximum amount of uncertainty, or, else, contains the minimum
amount of information. }

\textit{Entropy $S$, temperature $T$, and internal energy $U$ are related
via }

\textit{%
\begin{equation}  \label{relTH}
\frac{d S}{d U} = \frac{1}{T}.
\end{equation}
}

\textit{The corresponding expression in Thermodynamic Formalism is
Proposition 42 in \cite{LR}. }

\textit{Another related version of the Second Law can be described in the
following way according to \cite{Cale}: }

\textit{``Time does not occur as a variable in thermodynamic equations...
However, we are crucially interested in the direction of time, in the sense
of the distinction between the past and the future. The second law of
thermodynamics says that other variables being held constant heat always
flow from an object of higher temperature to an object of lower
temperature, and never the other way around.'' }

\textit{There is a discussion in the physics community if the  Boltzmann's entropy and the Gibbs' entropy are in fact the same (see the interesting discussion of the subject presented in the recent book \cite{Bricm}).}

\textit{The meaning of equilibrium can be tricky: in the case of glass,
recently, researchers contradicted the flowing glass window claim, by
determining that the glass in medieval windows only succumbs to gravity,
after very long geological time scales. Glass can be seen as in equilibrium,
or not, depending on the scale of time (see \cite{Bricm} for more details).
}

\smallskip

An expression of the relation of temperature with the variation of entropy
in Thermodynamic Formalism is described in Proposition 42 in \cite{LR}.

\smallskip

\textbf{Third Postulate:} The entropy of a composite system is additive over
each one of the individual components; the so-called extensive property of
entropy. Entropy is a continuous and monotonically increasing function of
energy.

\smallskip

\textit{Comment: From a dynamical point of view, the notion of
entropy in the above postulate is aimed at a system with discrete
time shift-invariance (stationary in $\mathbb{N}$)). However, in
Statistical Mechanics, this shift represents rather a shift in space,
typically the shift of spins in a chain. Thus, physically, entropy
usually refers to translation invariant states. Hence, we can also talk
about the entropy of states out of equilibrium, which, in particular, may
depend on time $t\in \mathbb{R}$. By contrast, usually, the
Shannon-Kolmogorov entropy is defined for invariant probabilities
(stationary in $\mathbb{N}$) and is a concept independent on time $t\in
\mathbb{R}$. The claim about additivity over a composite system is in the
sense that the entropy associated with a system described by two independent
systems is the sum of the entropy of each component; and this is true
for such kind of entropy.} The so-called non-extensive point of view
of Statistical Mechanics considers other concepts of entropy that are not
additive. \smallskip

\textbf{Fourth Postulate:} (Nernst law) The entropy vanishes when there is
only one equilibrium state at the absolute zero of temperature.

\textit{Comment: this postulate corresponds in some sense to thermodynamic
formalism to the property that for a generic H\"older potential $A$, the
ground state (a maximizing probability) is realized by a unique probability
with support in a periodic orbit, which, of course, has zero entropy (see
\cite{Con}, \cite{CLT} and \cite{LM3}). The Nernst Law is sometimes called
the Third Law.}

\medskip

After the above, we believe it is appropriate now to briefly describe the
point of view of the MaxEnt Method.

\begin{remark}
\label{hhy} Given $A:\{1,2,...,d\}\to \mathbb{R}$, and $\alpha \in \mathbb{R}
$, the $(A,\alpha)$-MaxEnt solution is the the vector of probability $\bar{p}%
= (\bar{p}_1,\bar{p}_2,...,\bar{p}_d)$ maximizing

\begin{equation*}
\max_{p=(p_1,p_2,...,p_d)} \{-\sum_{j=1}^d \log p_j\, p_j \,\,| \,\,
\sum_{j=1}^d p_j A(j) \,=\,\alpha\}=
\end{equation*}
\begin{equation}  \label{maxu}
\max_{p=(p_1,p_2,...,p_d)} \{h(p) \,\,| \,\, \sum_{j=1}^d p_j A(j)
\,=\,\alpha\} .
\end{equation}

Via Legendre transform one can show that there exists $\beta \in \mathbb{R}$%
, such that, $\bar{p}= (\bar{p}_1,\bar{p}_2,...,\bar{p}_d)$ also maximizes
the variational problem

\begin{equation*}
\max_{p=(p_1,p_2,...,p_d)} \{-\sum_{j=1}^d \log p_j\, p_j + \beta\,
\sum_{j=1}^d p_j A(j) \}=
\end{equation*}
\begin{equation}  \label{preu0}
\max_{p=(p_1,p_2,...,p_d)} \{h(p) + \beta\, \sum_{j=1}^d p_j A(j) \}..
\end{equation}

The above is equivalent to considering the variational problem (minimizing
\textit{free energy})

\begin{equation*}
\min_{p=(p_1,p_2,...,p_d)} \{\frac{1}{ \beta} \sum_{j=1}^d \, p_j \,\log
p_j\, +\, \, \sum_{j=1}^d p_j A(j) \}=
\end{equation*}
\begin{equation}  \label{preu}
\min_{p=(p_1,p_2,...,p_d)} \{ \, \sum_{j=1}^d p_j A(j)\,-\,\frac{1}{ \beta}
h(p)\, \} .
\end{equation}

Taking $\rho$ as the counting measure in $\{1,2,...,d\} $, one can show (see
for instance Proposition 7.5 in \cite{Bricm}) that the $(A,\alpha)$-MaxEnt
solution probability $\mu_{A,\alpha}$ in $\{1,2,...,d\} $ (the $\bar{p}$
maximal solution of \eqref{maxu}) can be written on the \textit{canonical
distribution} form
\begin{equation}  \label{uite}
B \subset \{1,2,...,d\} \to \mu_{A,\alpha}(B)= \frac{\int_B e^{- \beta A(x)}
d \rho (x)}{ \int_B e^{- \beta A(x)} d \rho (x)}= \frac{\sum_{j\in B} e^{-
\beta A(j)} p_j}{ \sum_{j=1}^d e^{- \beta A(j)} p_j}.
\end{equation}
For dynamical counterparts of the above see \cite{LR}, \cite{Lal2} and \cite%
{CL}.
\end{remark}

The postulates of equilibrium thermodynamics are an issue subject to
controversy. Indeed, we quote G. Gour in \cite{Gour}.

\smallskip

``Thermodynamics is one of the most prevailing theories in physics with vast
applications spreading from its early days focused on steam engines to modern
applications in biochemistry, nanotechnology, and black hole physics, just
to name a few. Despite the success of this field, the foundations of
thermodynamics remain controversial even today. Not only is there persistent
confusion over the relation between the macroscopic and microscopic laws, in
particular their reversibility and time-symmetry, there is not even a
consensus on how best to formulate the second law. Indeed, as the Nobel
laureate Percy Bridgman remarked in 1941 ``there are almost as many
formulations of the Second Law as there have been discussions of it'' and
the situation hasn't improved much since then.''

\begin{remark}
\label{dance} The natural variational problem in thermodynamics corresponds
to find the state (ensemble) $\rho_U$ minimizing the \textit{Helmholtz free
energy} (see (3.51) page 51 in \cite{Salinas})
\begin{equation}  \label{opseli}
F(T,V,N) := U - T \, S \;.
\end{equation}

The equilibrium probability (ensemble) $\rho_U$ for the Hamiltonian $U$ at
temperature $T$ is the one minimizing the integral
\begin{equation}  \label{opseli1}
\int ( U - T \, S)\, d \rho.
\end{equation}

For an illustration of this kind of variational problem see \eqref{preu}.

For $T $ fixed (or, equivalently for $\beta$ fixed), this corresponds to
find $\rho$ maximizing the integral of
\begin{equation}  \label{ops}
- \, \frac{F}{T} = - k_B \, \beta \,F= - k_B \, \beta U + S \;.
\end{equation}

For an illustration of this kind of variational problem see \eqref{preu0}.
\end{remark}

The rule of Nature determining a minimization of the expression %
\eqref{opseli1} can be seen as the connection of the microscopical variables
with the macroscopic variables. Actually, when $U$ is the Hamiltonian, to
maximize expression \eqref{ops} corresponds in Thermodynamic Formalism to
maximize topological pressure for the potential $- k_B \, \beta U$ among $%
\sigma$-invariant probabilities.

From the physical point of view, it is natural to introduce the chemical
potential $\mu $ and the macrostate variable $N$, which describes the number
of particles. The number of particles $N$ can range in the set of natural
numbers $\mathbb{N}$.

The \textit{grand-canonical thermodynamic potential} (see (3.55) page 52 in
\cite{Salinas}), is given by the expression
\begin{equation}
U(T,\mu ):=U-TS-\mu N\;,  \label{fun13}
\end{equation}%
where $\mu $ is the chemical potential.

In this case, the rule of Nature determines equilibrium via the minimization
of
\begin{equation}
\int ( U - TS - \mu N)\, d \rho,  \label{fun13}
\end{equation}
among probabilities $\rho$.

Physical experiments in the laboratory indicate that in several cases the
probability that the number of particles $N$ is large is very small; the
term $-\mu N$ in the above equation is in consonance with this claim.

\bigskip

\subsection{Grand-canonical systems - indefinite number of particles}

\label{bobo}

Now, we will outline a simplified version (suitable for us) of the main
issues on the topic of the grand-canonical partition as presented in chapter
1.6 in \cite{MR2777415} (see also Section 8 in \cite{Tsch}); here we
consider a certain number of indistinguishable particles $N$ ranging in $%
\mathbb{N}_{0}=\{0,1,2,...,N,...\}$. In dealing with particle densities, it
makes sense to consider the number of particles (per unit volume) as being]
a real non-negative number but we will avoid this issue here (see \cite{LW}
for a more detailed account).

Given $\beta>0$, and a sequence $A_N > 0$, $N \in \mathbb{N}_0$, we denote
\begin{equation*}
Z_N(\beta) := e^{- \beta A_N} \;.
\end{equation*}

So, for $\mu<0$, the \textit{grand-canonical partition sum} associated to
the family $\Phi = (A_N)_{N \in \mathbb{N}_0}$ is
\begin{equation}  \label{averq}
Z(\beta,\mu) := \sum_{N \in \mathbb{N}_0} e^{\beta\, N\,\mu} Z_N(\beta) =
\sum_{N \in \mathbb{N}_0} e^{\beta\, N\,\mu} e^{- \beta A_N} \;.
\end{equation}

Moreover, given $N \in \mathbb{N}_0$, the probability $\mathit{P}_{N,
\beta,\mu}$ of the number of particles to be $N$ at temperature $T = \frac{1%
}{k_B \beta}$, is given by the value
\begin{equation}  \label{aver}
\mathit{P}_{N, \beta, \mu} := \frac{e^{\beta\, N\,\mu} e^{- \beta A_N}}{%
Z(\beta,\mu) } \;.
\end{equation}

Compare \eqref{aver} with \eqref{hu1}.

\smallskip

Besides that, important information is given by the partial derivatives
\begin{equation*}
\frac{\partial }{\partial \beta} \log(Z(\beta,\mu))|_{\beta=\beta_0} =\mu
\,\sum_{N \in \mathbb{N}_0} N \, \frac{e^{\beta_0 \,N\,\mu} e^{- \beta_0 A_N}%
}{ Z(\beta_0,\mu)}- \sum_{N \in \mathbb{N}_0} A_N \,\frac{ e^{\beta_0\,
N\,\mu} e^{- \beta_0 A_N}}{ Z(\beta_0,\mu)} =
\end{equation*}
\begin{equation}  \label{first-derivative}
\mu\,\sum_{N \in \mathbb{N}_0} <N>_{ \mathit{P}_{N, \beta_0, \mu}} - \sum_{N
\in \mathbb{N}_0} <A_N>_{ \mathit{P}_{N, \beta_0, \mu}}.
\end{equation}
and
\begin{equation}  \label{second-derivative}
\frac{\partial}{\partial \mu} \log(Z(\beta,\mu))|_{\mu=\mu_0}= \beta
\,\sum_{N \in \mathbb{N}_0} N \, \frac{e^{\beta\, N\,\mu_0} e^{- \beta A_N}}{%
Z(\beta,\mu_0) } = \beta\,\sum_{N \in \mathbb{N}_0} <N>_{ \mathit{P}_{N,
\beta, \mu_0}}.
\end{equation}

The two above expressions are analogous to the ones appearing in equations
(1.37) and (1.38) in \cite{MR2777415}.
\medskip

Now, we want to trace a parallel with the case of finite particles
(comparing with \eqref{hh1} etc,...)

Given $\beta,A:\{1,2,...,d\} \to \mathbb{R}$, the \textit{microcanonical
partition function (see \eqref{hh1}) is}

\begin{equation}  \label{hh541}
Z (\beta) = \sum_{j=1}^d \, e^{- \beta A(j)}.
\end{equation}

As we mentioned before, the \textit{canonical distribution} $%
\mu^{can}=\mu_{A, \beta}^{can}$ is the probability $\mathfrak{Q}$ such that
\begin{equation}  \label{hh132}
\,\,\,j_0 \to\,\mathfrak{Q}(j_0)=\mu^{can}(j_0)=\frac{ e^{- \beta A(j_0)} }{
\sum_{j=1}^d e^{- \beta A(j)}}= \frac{ e^{- \beta A(j_0)} }{ Z(\beta)}.
\end{equation}

Now we present the dynamical version of \eqref{hh1} and \eqref{hh132}: given
a Lipschitz function $f:\Omega \to \mathbb{R},$ and $x_0 \in \Omega$, denote
\begin{equation*}
\lambda= e^{P (- \beta A)}.
\end{equation*}

It is known (see \cite{MR1085356}) that
\begin{equation*}
\lim_{n \to \infty} \frac{\mathcal{L}_{- \beta A}^n (f) (x_0)}{ \lambda^n}=
\varphi(x_0)\, \int f d \nu,
\end{equation*}
where $\varphi$ is the eigenfunction and $\nu$ the eigenprobability of the
Ruelle operator $\, \mathcal{L}_{- \beta A}$.

Then, the $Z (\beta)$ in expression \eqref{hh1} correspond here to $%
\lambda=e^{P (- \beta A)}.$

In our setting it is natural to denote $Z(\beta):= e^{P(- \beta A) }$.

Then, $P(\beta ):=\log Z(\beta )$ (see also\eqref{main2124}) \thinspace\ $%
\Leftrightarrow $\thinspace\ Topological Pressure \thinspace $P(-\beta A)$.

We believe we made more clear to the reader the relationship between the
corresponding concepts under the two possible settings (more details in
Remark \ref{rel}).

\medskip

In the case of the statistics of a countable number of particles, it is a
different setting: we consider a number of indistinguishable particles $N$
ranging in $\mathbb{N}_0$.

Remember that $\mu<0$ is called the chemical potential. Therefore, it is
natural in our dynamical setting to adapt the reasoning which derived %
\eqref{hh1}.

\smallskip

Given $\beta>0$, $\mu<0$ and a sequence $A_N \geq 0$, $N \in \mathbb{%
N}_0$, the \textit{grand-canonical partition sum} (compare with \eqref{hh541}%
) is
\begin{equation}  \label{ppo}
Z(\beta,\mu) := \sum_{N \in \mathbb{N}_0} e^{\beta\, N\,\mu} e^{- \beta A_N}
\;.
\end{equation}

An example: if each particle has energy $E>0$, then, take $A_N=N\,E.$
\medskip

Moreover, given $N \in \mathbb{N}_0$, the probability $\mathit{P}_{N,
\beta,\mu}$ of the number of particles to be $N$ at temperature $T = \frac{1%
}{k_B \beta}$, is (see \eqref{aver})
\begin{equation}  \label{ave24r}
\mathit{P}_{N, \beta, \mu} := \frac{e^{\beta\, N\,\mu} e^{- \beta A_N}}{%
Z(\beta,\mu)}\,\text{ (compare with\,} \eqref{hh132}\,\mu^{can}(j_0)=\frac{
e^{- \beta A(j_0)} }{ Z(\beta)} ).
\end{equation}

In consonance with \eqref{1sobreNlnBN} and \eqref{hh1} we call
\begin{equation}  \label{ave2123}
P(\beta,\mu)= \log Z(\beta,\mu))
\end{equation}
grand-canonical asymptotic pressure. \medskip

\begin{remark}
\label{rel} Fixing the volume $V$, the chemical potential $\mu $, the
temperature $T$, and the \textit{grand-canonical pressure}, we get that the
grand-canonical \textit{gas pressure} $p$ satisfies (see page 12 in \cite%
{MR2777415}).

\smallskip

\begin{equation}  \label{main2123}
p := \frac{k_B \,T\, P(\beta, \mu)}{V}\,\,\,\,\,\,\text{(compare with}\,\, %
\eqref{fun1} \,\,\,p\, =\frac{ k_B\, T\, N}{ V}).
\end{equation}

\smallskip

If $V$ and $T$ are fixed, then $p$ is linear on $P(\beta ,\mu ),$ or
considering the case of the statistics of just one particle, $\mu =0$, $V=1$
and $T=1$, we get
\begin{equation}
p\sim k_{B}P(\beta ,\mu )\sim \, k_{B}P(\beta ).  \label{main2124}
\end{equation}

In this way, the origin of the terminology topological pressure can be seen
as related in some way to the pressure of a gas.
\end{remark}

Below, we will consider a non-dynamical example:


\begin{example}
\label{Ham-cte}  Consider a chemical potential $\mu < 0$, assume that $E > 0$ represents the energy of a particle, and, under the assumption of non-interaction between particles, we take the Hamiltonian $A_N$  of the form  $A_N= N E$, for
each $N \in \mathbb{N}_0$. In this case, we have 
\begin{equation*}
Z(\beta,\mu) = \sum_{N \in \mathbb{N}_0} e^{\beta\, N\,\mu} e^{- \beta N\,
E}= \sum_{N \in \mathbb{N}_0} e^{N\, \beta\, \,(\mu - E)} = \frac{1}{ 1-
e^{\beta (\mu- E)}} > 0 \;. 
\end{equation*}

Moreover, under the above assumptions, it follows that 
\begin{align*}
\frac{\partial}{\partial \beta} \log(Z(\beta,\mu))|_{\beta=\beta_0} &= \frac{%
1}{Z(\beta_0, \mu)}\sum_{N \in \mathbb{N}} N(\mu - E)e^{N\, \beta_0\, \,(\mu
- E)} \\
&= \frac{(\mu - E)e^{\beta_0(\mu - E)}}{Z(\beta_0, \mu)(1 - e^{\beta_0(\mu -
E)})^2} \;,
\end{align*}
and 
\begin{align*}
\frac{\partial }{\partial \mu} \log(Z(\beta,\mu))|_{\mu=\mu_0} &= \frac{1}{%
Z(\beta, \mu_0)} \sum_{N \in \mathbb{N}} N\, \beta e^{N\, \beta\, \,(\mu_0 -
E)} \\
&= \frac{\beta e^{\beta(\mu_0 - E)}}{Z(\beta, \mu_0)(1 - e^{\beta(\mu_0 -
E)})^2} \;.
\end{align*}

\hspace{12.6cm} $\diamondsuit$
\end{example}
\medskip

On behalf of all authors, the corresponding author states that there is no conflict of interest

\medskip

\medskip

A. O. Lopes, IME-UFRGS. Part. supported by CNPq.

E. R. Oliveira, IME-UFRGS

W. de S. Pedra, ICMC-USP. Part. supported by CNPq (309723/2020-5).

V. Vargas, Center for Mathematics - University of Porto. Supported by FCT
Project UIDP/00144/2020.

\medskip


\medskip



\medskip


\begin{thebibliography}{GKLM18}
\bibitem[ARU18]{Urb0} Atnip, J., Roychowdhury, M. K., Urba\'{n}ski, M.. %
\newblock Quantization dimension for infinite conformal iterated function
systems. \newblock {\em arXiv:1811.00299v1} (2018).

\bibitem[BCLMS23]{BCLMS23} Baraviera, A. T., Cioletti, L. M., Lopes, A. O., Mohr, J.,  Souza, R. R. (2011). On the general one-dimensional XY model: positive and zero temperature, selection and non-selection. Reviews in Mathematical Physics, 23(10), 1063-1113.

\bibitem[BLL13]{MR3114331} Baraviera, A., Leplaideur, R., Lopes, A. O..
\newblock {\em Ergodic optimization, zero temperature limits and the max-plus
  algebra}. \newblock Publica\c{c}\~{o}es Matem\'{a}ticas do IMPA. \newblock %
29th Brazilian Mathematics Colloquium (2013)

\bibitem[BS95]{BeSch} Beck, C., Schlogl, F..
\newblock {\em Thermodynamics
of chaotic systems}, \newblock Cambridge University Press, Cambridge (1995)

\bibitem[Bai12]{Bai} Baierlein, R.. \newblock {\em Thermal Physics}, %
\newblock Cambridge University Press, Cambridge (2012)

\bibitem[Bena]{Bena} Benatti, F.. Dynamics, Information and Complexity in
Quantum Systems. Springer (2009)

\bibitem[BBE]{BBE} Beltran, E. R., Bissacot, R., Endo, E. O.. \newblock %
Infinite DLR measures and volume-type phase transitions on countable Markov
shifts. \newblock Nonlinearity 34, no. 7, 4819--4843 (2021)


\bibitem[Bowen]{Bowen} Bowen, R.. Equilibrium states
and the ergodic theory of Anosov diffeomorphisms. Lecture Notes in Mathematics, vol. 470 (1975)

\bibitem[BOS]{BOS} Brasil, J., Oliveira, E. R., Souza, R.. Thermodynamic
formalism for general iterated function systems with measures, Qualitative
Theory of Dynamical Systems, volume 22, Article number: 19 (2023)

\bibitem[Bri22]{Bricm} Bricmont, J..
\newblock {\em Making Sense of
Statistical} \newblock Mechanics, Springer Verlag, (2022).

\bibitem[Cal14]{Cale} Calegari, D..
\newblock {\em Notes on Thermodynamic
Formalism}, \newblock Lecture Notes, University of Chicago, (2014).

\bibitem[Cat08]{Cat} Caticha, A.. \newblock Lectures on probability, entropy
and statistical physics, Entropic Physics, \newblock {\em arXiv:0808.0012},
(2008).

\bibitem[CL17]{CL17} Cioletti, L., Lopes, A. O.. \newblock Interactions,
Specifications, DLR probabilities and the Ruelle Operator in the
One-Dimensional Lattice, \newblock {\em Discrete Contin. Dyn. Syst.},
37(12):6139--6152 (2017)

\bibitem[CHLS]{CHLS}  Cioletti, L., Hataishi, L., Lopes, A. O., Stadlbauert, M..
Spectral Triples on Thermodynamic Formalism and Dixmier Trace Representations of Gibbs Measures: theory and examples,
arXiv (2022)


\bibitem[CO17]{CO17} Cioletti, L., Oliveira, E. R.. \newblock  Thermodynamic
Formalism for Iterated Function Systems with Weights, \newblock
{arXiv:1707.01892v1} (2017)

\bibitem[CLT01]{CLT} Contreras, G., Lopes, A. O., Thieullen. Ph.. \newblock
Lyapunov minimizing measures for expanding maps of the circle. \newblock
\emph{Ergodic Theory Dynam. Systems}, 21(5):1379--1409 (2001)

\bibitem[Con16]{Con} Contreras, G.. \newblock Ground states are generically
a periodic orbit. \newblock {\em Invent. Math.}, 205(2):383--412 (2016)

\bibitem[CL00]{CL} Craizer, M., Lopes, A. O.. \newblock The capacity
costfunction of a hard constrained channel.
\newblock {\em Int. Journal of
Appl. Math.} 2(10):1165--1180 (2000)

 \bibitem[Fan0]{Fan0} Fan, A. H.. A proof of the Ruelle operator theorem
Fan, A. H.
Rev. Math. Phys. 7 (1995), no. 8, 1241--1247.

\bibitem[FL99]{Fan} Fan, A. H., Lau, K-S.. \newblock Iterated
function system and Ruelle operator. \newblock {\em J. Math. Anal. Appl.}
231(2):319--344 (1999)

\bibitem[FV]{FV} Freire, R., Vargas, V.. \newblock Equilibrium states and
zero temperature limit on topologically transitive countable markov shifts. %
\newblock  Trans. Amer. Math. Soc., 370, no. 12, 8451--8465 (2018)

\bibitem[Gal99]{Gal} Gallavotti, G.. \newblock {\em Statistical Mechanics}. %
\newblock Springer Verlag, Berlin, Heidelberg (1999)

\bibitem[GKLM18]{MR3852182} Giulietti, P., Kloeckner, B., Lopes, A. O.,
Marcon, D.. \newblock The calculus of thermodynamical formalism. \newblock
\emph{J. Eur. Math. Soc. (JEMS)}, 20(10):2357--2412 (2018)

\bibitem[Gou22]{Gour} Gour, G.. \newblock On the role of quantum coherence
in thermodynamics. \newblock {\em PRX Quantum 3, 040323} (2022)

\bibitem[HMU02]{Urb} Hanus, P., Mauldin, R., Urba\'{n}ski, M.. \newblock %
Thermodynamic formalism and multifractal analysis of conformal infinite
iterated function systems. \newblock {\em Acta Math. Hungar.} 96.1-2 : 27-98
(2002)

\bibitem[Henion]{He} Hennion, H.. Sur un Theorem Spectral et son application aux noyaux Lipchtziens, Proc. AMS Vol 118 - N2 - 627-634 (1993)

\bibitem[ITM50]{MR37469} Ionescu-Tulcea, C. T., Marinescu, G.. \newblock Th%
\'{e}orie ergodique pour des classes d'op\'{e}rations non compl\`etement
continues. \newblock {\em Ann. of Math.}, 52(2):140--147 (1950)

\bibitem[ILM18]{LM3} Iturriaga, R., Lopes, A. O., Mengue, J. K.. \newblock %
Selection of calibrated subaction when temperature goes to zero in the
discounted problem. \newblock {\em Discrete Contin. Dyn. Syst.}
38(10):4997--5010 (2018)

\bibitem[Lal87]{Lal2} Lalley, S.. \newblock Distribution of Periodic Orbits
of Symbolic and Axiom A Flows. \newblock {\em Adv. in Appl. Math.},
8:154--193 (1987)

\bibitem[Lop1] {Lop1} Lopes, A. O.. Thermodynamic Formalism, Maximizing
Probabilities and Large Deviations, Notes on line - UFRGS

http://mat.ufrgs.br/$\sim$alopes/pub3/notesformteherm.pdf


\bibitem[LR22]{LR} Lopes, A. O., Ruggiero, R.. \newblock Nonequilibrium in
Thermodynamic Formalism: the second law, gases and information geometry. %
\newblock {\em Qual. Theory Dyn. Syst.}, 21(1):44, \newblock Id/No 21 (2022)

\bibitem[LO09]{LO} Lopes, A. O., Oliveira, E. R.. \newblock Entropy and
variational principles for holonomic probabilities of IFS.
\newblock {\em
Discrete Contin. Dyn. Syst.} 23(3):937 (2009)

\bibitem[LW]{LW} Lopes, A. O., de Siqueira Pedra, W.. A Friendly
Introduction to Thermodynamics, Statistical Physics in Relation to the
Thermodynamic Formalism of Dynamical Systems, \newblock in preparation


\bibitem[Mane]{Mane}  Ma\~n\'e, R.. Notes on Thermodynamic Formalism (2023)

http://mat.ufrgs.br/$\sim$alopes/pub3/Thermodynamic-Mane-2023.pdf

\bibitem[Mih22]{Mih} Mihailescu, E.. \newblock Thermodynamic formalism for
invariant measures in iterated function systems with overlaps. \newblock
\emph{Commun. Contemp. Math.} 24.06: 2150041 (2022)

\bibitem[Nau11]{MR2777415} Naudts, J..
\newblock {\em Generalised
thermostatistics}. \newblock Springer-Verlag London, Ltd., London (2011)

\bibitem[PP90]{MR1085356} Parry, W., Pollicott, M.. \newblock Zeta functions
and the periodic orbit structure of hyperbolic dynamics. \newblock {\em
Ast\'erisque}, (187-188):268 (1990)

\bibitem[Sala]{Sala} Salamon, P., Andresen, B., Nulton, J., Konokpa, K.. %
\newblock The mathematical structure of thermodynamics, \newblock  published
on ``Systems Biology: Principles, Methods, and Concepts'' Edition: 1,
Chapter: 9, Taylor and Francis/CRC Press (2006)

\bibitem[Sal01]{Salinas} Salinas, S. R. A..
\newblock {\em Introduction to
statistical physics}. \newblock Graduate Texts in Contemporary Physics.
Springer-Verlag, New York, \newblock Translated from the Portuguese (2001)

\bibitem[Sarig]{Sarig} Sarig, O.. \newblock Thermodynamic formalism for
countable Markov shifts. \newblock
Ergodic Theory Dynam. Systems, 19(6):1565-1593 (1999)

\bibitem[Sch89]{Sch} Schl\"{o}gl, F.. \newblock {\em Probability and Heat}. %
\newblock Graduate Texts in Contemporary Physics. Vieweg+Teubner Verlag
Wiesbaden, \newblock Fundamentals of Thermostatistics (1989)


\bibitem[SilvaFe]{SilvaFe} Silva, E. A., Dimens\~ao de Hausdorff de Ferraduras, Master dissertation - UNB - Brasilia (2010)


\bibitem[Ste01]{Ste01} Stenflo, O.. A note on a theorem of Karlin. Statistics \& probability letters, v. 54, n. 2, p. 183-187, 2001.

\bibitem[Tsch]{Tsch} Tschoegl, N. W.. Fundamentals of Equilibrium and
Steady-State Thermodynamics, Elsevier (2000)

\bibitem[VL]{MR3309581} von~der Linden, W., Dose, V., and von Toussaint, U..
\newblock {\em Bayesian probability theory: Applications in the physical
  sciences}. \newblock Cambridge University Press, Cambridge (2014)

\bibitem[Wal07]{W}  Walters, P.. A natural
space of functions for the Ruelle operator theorem, Ergodic Theory and Dynamical Systems, 27,
1323-1348 (2007).
\bibitem[Ye]{Ye}
Ye,Y-L., Ruelle operator with weakly contractive iteration function systems, Erg. Theo and DYn. Syst. 33 - 1265-1290 (2013)

\bibitem[Zu]{Zu} Zubarev, D.. Nonequilibrium Statistical Thermodynamics,
Springer (1974)

\end{thebibliography}
\end{document}